\documentclass{amsart}

\usepackage{amsmath,amsthm,amssymb,bm}
\usepackage{hyperref}
\usepackage{graphicx,color}
\usepackage{tikz}
\usetikzlibrary{snakes}
\usepackage{ytableau}
\numberwithin{equation}{section}

\newtheorem{thm}{Theorem}[section]
\newtheorem{lem}[thm]{Lemma}
\newtheorem{prop}[thm]{Proposition}
\newtheorem{cor}[thm]{Corollary}

\theoremstyle{definition}

\newtheorem{defn}[thm]{Definition}

\newtheorem{remark}[thm]{Remark}

\newcommand\sgn{\operatorname{sgn}}
\newcommand\NI{\operatorname{NI}}

\newcommand\pwt{\operatorname{pwt}}
\newcommand\Sch{\operatorname{Sch}}
\newcommand\Alt{\operatorname{Alt}}

\newcommand\qand{\quad\mbox{and}\quad}
\newcommand\PV{\operatorname{PV}}

\newcommand\vz{\bm{0}}
\newcommand\vo{\bm{1}}
\newcommand\vb{\bm{b}}

\newcommand\va{\bm{a}}
\newcommand\vla{\bm{\lambda}}
\newcommand\vpi{\bm{\pi}}
\newcommand\vI{\bm{I}}
\newcommand\vJ{\bm{J}}

\newcommand\LL{\mathcal{L}}
\newcommand\ZZ{\mathbb{Z}}
\newcommand\RR{\mathbb{R}}
\newcommand\Mot{\operatorname{Mot}}
\newcommand\wt{\operatorname{wt}}
\newcommand\flr[1]{\left\lfloor #1\right\rfloor}
\newcommand\norm[1]{\Vert #1\Vert}
\newcommand\RPP{{\operatorname{RPP}}}

\title{Negative moments of orthogonal polynomials}
\date{\today}

\author{Jihyeug Jang}
\address{Department of Mathematics, Sungkyunkwan University, Suwon,
  South Korea}
\email{4242ab@gmail.com}
\author{Donghyun Kim}
\address{Department of Mathematics, Sungkyunkwan University, Suwon,
  South Korea}
\email{hyun923010@g.skku.edu}
\author{Jang Soo Kim}
\address{Department of Mathematics, Sungkyunkwan University, Suwon,
  South Korea}
\email{jangsookim@skku.edu}
\author{Minho Song}
\address{Department of Mathematics, Sungkyunkwan University, Suwon,
  South Korea}
\email{smh3227@skku.edu}
\author{U-Keun Song}
\address{Department of Mathematics, Sungkyunkwan University, Suwon,
  South Korea}
\email{sukeun319@gmail.com}

\keywords{orthogonal polynomials, combinatorial reciprocity, lattice paths, continued fractions,
determinants}
\subjclass[2020]{Primary: 05A15; Secondary: 05A19}

\begin{document}

\begin{abstract}
  If a sequence indexed by nonnegative integers satisfies a linear recurrence
  without constant terms, one can extend the indices of the sequence to negative
  integers using the recurrence. Recently, Cigler and Krattenthaler showed that
  the negative version of the number of bounded Dyck paths is the number of
  bounded alternating sequences. In this paper we provide two methods to compute
  the negative versions of sequences related to moments of orthogonal
  polynomials. We give a combinatorial model for the negative version of the
  number of bounded Motzkin paths. We also prove two conjectures of Cigler and
  Krattenthaler on reciprocity between determinants.
\end{abstract}

\maketitle

\section{Introduction}

Suppose that there is a sequence \( (f_n)_{n\in\ZZ} \) indexed by all integers.
If both \( |f_n| \) and \( |f_{-n}| \) count some combinatorial objects of size
\( n\ge1 \), such a result is called a \emph{combinatorial reciprocity theorem}, a
term first used by Richard Stanley \cite{Stanley74}. There are many
combinatorial reciprocity theorems; three notable examples are when \( f_n \) is
the binomial coefficient \( \binom{n}{k} \), the chromatic polynomial \(
\chi_G(n) \) of a graph \( G \), and the Ehrhart polynomial \( \mathrm{Ehr}_P(n)
\) of a lattice polytope \( P \). For more details on combinatorial reciprocity
theorems, see the book by Beck and Sanyal \cite{BeckSanyal}.

Suppose now that we have a sequence \( (f_n)_{n\ge0} \) indexed by nonnegative
integers. If the sequence satisfies a homogeneous linear recurrence relation, then one can extend the indices of this sequence to
negative integers \( (f_{-n})_{n\ge 1} \) using the recurrence. Recently, Cigler
and Krattenthaler \cite{Kratt_Hankel1} showed that, for a fixed integer \( k \),
the negative counterpart of the number of Dyck paths from \( (0,0) \) to \(
(2n,0) \) with bounded height \( 2k-1 \) is the number of alternating sequences
\( a_1\le a_2\ge a_3 \le \cdots \ge a_{2n-1} \) of positive integers at most \(
k \). They also showed many other interesting results including a reciprocity
between determinants of these numbers and their connection with orthogonal
polynomials.

In this paper, motivated by the work of Cigler and Krattenthaler
\cite{Kratt_Hankel1}, we find combinatorial reciprocity theorems for more
general sequences related to moments of orthogonal polynomials. In particular we
give two methods to study such negative sequences; the first method uses
continued fractions and the second one uses matrix inverses. Our first method is
new and the key idea of the second method is due to Hopkins and Zaimi~\cite{Hopkins2023}. We also
prove two conjectures on reciprocity between determinants proposed by Cigler and
Krattenthaler~\cite[Conjectures~50, 53]{Kratt_Hankel1}. Before stating our
results we first review basic results in orthogonal polynomials and define some
notation.

\medskip

A sequence \( (P_n(x))_{n\ge0} \) of polynomials is called an
\emph{orthogonal polynomial sequence}\footnote{In this paper we only
  consider the ``formal'' orthogonality in the sense that we do not
  require the positive-definiteness of the linear functional
  \( \LL \), which is often assumed in the literature on orthogonal
  polynomials.} with respect to a linear functional \( \LL \) if for
all \( m,n\ge0 \), we have \( \deg(P_n(x))=n \) and
\begin{equation}\label{eq:orthogonality}
  \LL(P_m(x)P_n(x)) = \delta_{m,n}K_n, \qquad K_n\ne0.
\end{equation}
In this case we will simply say that \( P_n(x) \) are \emph{orthogonal
  polynomials} (with respect to \( \LL \)).

It is well known \cite[Theorem~4.1, p.18]{Chihara} that monic
orthogonal polynomials \( P_n(x) \) satisfy a three-term recurrence
relation:
\begin{equation}\label{eq:OP}
  P_{n+1}(x) = (x-b_n)P_n(x) -\lambda_n P_{n-1}(x),\quad n\ge0, \qquad
  P_{-1}(x)=0,P_0(x)=1,
\end{equation}
for some sequences \( \vb=(b_n)_{n\ge0} \) and
\( \vla=(\lambda_n)_{n\ge1} \) with \( \lambda_n\ne0 \). Conversely,
Favard's theorem \cite[Theorem~4.4, p.21]{Chihara} states that if
monic polynomials \( P_n(x) \) satisfy \eqref{eq:OP} for some
sequences \( \vb=(b_n)_{n\ge0} \) and \( \vla=(\lambda_n)_{n\ge1} \)
with \( \lambda_n\ne0 \), then \( P_n(x) \) are orthogonal polynomials
with respect to a unique linear function \( \LL \) satisfying
\eqref{eq:orthogonality} and \( \LL(1)=1 \).

Let \( P_n(x;\vb,\vla) \) denote the polynomials satisfying
\eqref{eq:OP}. Then by Favard's theorem these are orthogonal
polynomials with respect to a unique linear functional \( \LL \). The
\emph{moment} \( \mu_{n}(\vb,\vla) \) of the orthogonal polynomials
\( P_n(x;\vb,\vla) \) is defined by \(\mu_{n}(\vb,\vla) = \LL(x^n) \).

Viennot \cite{ViennotLN} found the following combinatorial interpretation
for the moment:
\[
  \LL(x^n) =  \mu_n(\vb,\vla)= \sum_{p\in \Mot_{n}} \wt(p;\vb,\vla),
\]
where \( \Mot_{n} \) is the set of Motzkin paths from \( (0,0) \) to \(
(n,0) \) and \( \wt(p;\vb,\vla) \) is a weight of a Motzkin path \( p
\) depending on the sequences \( \vb \) and \( \vla \). See
Section~\ref{sec:preliminaries} for the precise definitions.

We define the \emph{bounded moments} \( \mu_n^{\le k}(\vb,\vla) \) by
\[
  \mu_{n}^{\le k}(\vb,\vla) := \sum_{p\in \Mot^{\le k}_{n}} \wt(p;\vb,\vla),
\]
where \( \Mot_{n}^{\le k}(\vb,\vla) \) is the set of Motzkin paths from \(
(0,0) \) to \( (n,0) \) that stay weakly below the line \( y=k \). Then the
moments are the limits of the bounded moments:
\[
  \LL(x^n) =  \mu_n(\vb,\vla) = \lim_{k\to\infty} \mu_{n}^{\le k}(\vb,\vla).
\]

For certain
choices of \( \vb,\vla\), and \( k \), the sequence \( (\mu_{n}^{\le
  k}(\vb,\vla))_{n\ge0}\) satisfies a homogeneous linear recurrence relation so
that its negative version \( (\mu_{-n}^{\le k}(\vb,\vla))_{n\ge1}\) is defined.
In this case we call \( \mu_{-n}^{\le k}(\vb,\vla) \) the \emph{negative (bounded)
  moments} of the orthogonal polynomials \( P_n(x;\vb,\vla) \).

\medskip

Cigler and Krattenthaler \cite{Kratt_Hankel1} showed the following combinatorial
reciprocity theorem for the number \( \mu_{2n}^{\le 2k-1}(\vz,\vo) \) of bounded
Dyck paths, where \( \vz=(0,0,\dots) \) and \( \vo=(1,1,\dots) \).

\begin{thm}\label{thm:CK}\cite[Corollary~13]{Kratt_Hankel1}
  For positive integers \( n \) and \( k \),
  \[
    \mu_{-2n}^{\le 2k-1}(\vz,\vo) = |\Alt_{2n-1}^{\le k}|,
  \]
  where \( \Alt_{n}^{\le k} \) is the set of alternating sequences \(
  (a_1,\dots,a_n) \) of integers such that \( a_1\le a_2\ge a_3 \le \cdots \)
  and \( 1\le a_i\le k \) for all \( i \).
\end{thm}

Cigler and Krattenthaler \cite{Kratt_Hankel1} proved Theorem~\ref{thm:CK} using
generating functions. We give a new proof of this theorem using continued fractions. To do this
we introduce a notion of \emph{\( \ell \)-peak-valley sequences} in
Definition~\ref{def:PV} and give a simple bijection between alternating
sequences and \( 2 \)-peak-valley sequences.

Using continued fractions we show in Theorem~\ref{thm:neg_mom b=0} that \(
\mu_{-2n}^{\le 2k-1} (\vz,\vla) \) is a weight generating function for \( 2
\)-peak-valley sequences with some conditions, which is equivalent to
\cite[Corollary~32]{Kratt_Hankel1}. Our method also applies to Motzkin paths. In
Theorems~\ref{thm:neg3k-1} and \ref{thm:neg3k} we show that if \( \vb \) and \(
\vla \) satisfy \( \lambda_i=b_{i-1}b_i \) for all \( i\ge1 \), then \(
\mu_{-n}^{\le 3k-1}(\vb,\vla) \) and \( \mu_{-n}^{\le 3k}(\vb,\vla) \) are
weight generating functions for \( 3 \)-peak-valley sequences with some
conditions.

\medskip

Viennot (see \cite[Proposition~17, p.~I.15]{ViennotLN} or
\cite[(5)]{ViennotOP}) also showed that the generalized moment
\( \mu_{n,r,s}(\vb,\vla) := \LL(x^n P_r(x;\vb,\vla) P_s(x;\vb,\vla))
\) has a similar combinatorial expression
\[
  \mu_{n,r,s}(\vb,\vla)= \sum_{p\in \Mot_{n,r,s}} \wt(p;\vb,\vla),
\]
where \( \Mot_{n,r,s}(\vb,\vla) \) is the set of Motzkin paths from \( (0,r) \)
to \( (n,s) \). We define the generalized bounded moments \(
\mu_{n,r,s}^{\le k}(\vb,\vla) \) by
\[
  \mu_{n,r,s}^{\le k}(\vb,\vla) := \sum_{p\in \Mot^{\le k}_{n,r,s}} \wt(p;\vb,\vla),
\]
where \( \Mot_{n,r,s}^{\le k}(\vb,\vla) \) is the set of Motzkin paths from \(
(0,r) \) to \( (n,s) \) that stay weakly below the line \( y=k \).

Cigler and Krattenthaler \cite{Kratt_Hankel1} showed that Theorem~\ref{thm:CK}
extends nicely to generalized bounded moments as follows.

\begin{thm}\label{thm:CK,r,s}\cite[Corollary~12]{Kratt_Hankel1}
  For positive integers \( n,k,r,s \) with \( 1\le r,s\le k \), we have
  \begin{align*}
    (-1)^{r+s}\mu_{-2n,2r-2,2s-2}^{\le 2k-1}(\vz,\vo) &= |\Alt_{2n+1,r,s}^{\le k}|,\\
    (-1)^{r+s}\mu_{-2n+1,2r-2,2s-1}^{\le 2k-1}(\vz,\vo) &= |\Alt_{2n,r,s}^{\le k}|,
  \end{align*}
  where \( \Alt_{n,r,s}^{\le k} \) is the set of sequences \( (a_1,\dots,a_n)
  \) of integers such that \( a_1\le a_2\ge a_3 \le \cdots \) and \( 1\le a_i\le
  k \) for all \( i \) and such that \( a_1=r \) and \( a_n=s \).
\end{thm}

In Theorems~\ref{thm:rsneg3k-1} and \ref{thm:rsneg3k} we show that if \( \vb \)
and \( \vla \) satisfy \( \lambda_i=b_{i-1}b_i \) for all \( i\ge1 \), then \(
\mu_{-n,r,s}^{\le 3k-1}(\vb,\vla) \) and \( \mu_{-n,r,s}^{\le 3k}(\vb,\vla) \)
are weight generating functions for \( 3 \)-peak-valley sequences with some
conditions.

\medskip

Cigler and Krattenthaler \cite{Kratt_Hankel1} showed the following reciprocity
theorem relating determinants whose entries are \( \mu_{n}^{\le 2k-1}(\vz,\vo) \)
and their negative versions, respectively.

\begin{thm}\label{thm:CK2}\cite[Theorem~15]{Kratt_Hankel1}
  For all nonnegative integers $n, k, m$, we have
  \[
    \det\left(\mu_{2 n+2 i+2 j+4 m-2}^{\le 2 k+2 m-1}(\vz,\vo)\right)_{i, j =0  }^{ k-1}
    =\det\left(\mu_{-2 n-2 i-2 j}^{\le 2 k+2 m-1}(\vz,\vo)\right)_{i, j =0  }^{ m-1}.
  \]
\end{thm}

Cigler and Krattenthaler \cite{Kratt_Hankel1} proposed the following two
conjectures.

\begin{thm}\label{thm:CKconjecture1}\cite[Conjecture~50]{Kratt_Hankel1}
  For all nonnegative integers $n, k, m$, we have
  \[
    \det\left(\sum_{s=0}^{2 k+2 m-1} \mu_{n+i+j+2 m-1,0,s}^{\le 2 k+2 m-1}(\vz,\vo)\right)_{i, j =0  }^{ k-1}
    =(-1)^{\left(\binom{k}{2}+\binom{m}{2}\right)(n+1)}
    \det\left(\left|\Alt_{n+i+j}^{\le k+m}\right|\right)_{i, j =0  }^{ m-1}.
  \]
\end{thm}

\begin{thm}\label{thm:CKconjecture2}
  \cite[Conjecture~53]{Kratt_Hankel1} 
  For all positive integers \( n, k, m \) with \( k+m \not \equiv 2 \pmod 3 \), we have
  \[
    \det\left(\mu_{n+i+j+2 m-2}^{\le k+m-1}(\vo,\vo)\right)_{i, j =0  }^{ k-1} \\
    =(-1)^{n\flr{(k+m) / 3}} \det\left(\mu_{-n-i-j}^{\le k+m-1}(\vo,\vo)\right)_{i, j =0  }^{ m-1}.
  \]
\end{thm}

In Section~\ref{sec:gener-recipr-theor} we prove a general reciprocity theorem
(Theorem~\ref{thm: CKconjecturemain}). In Section~\ref{sec:conj} we prove the
above two conjectures using Theorem~\ref{thm: CKconjecturemain}. In
Section~\ref{sec:appl-gener-recipr} we show that Theorem~\ref{thm:
  CKconjecturemain} also implies the weighted version of Theorem~\ref{thm:CK2}
due to Cigler and Krattenthaler \cite[Theorem~34]{Kratt_Hankel1}. We then show
in Theorem~\ref{thm:RPP} that this weighted version gives a bounded and
multivariate generalization of the Morales--Pak--Panova ex-conjecture
\cite{MPP2} on reverse plane partitions, which has been proved by Hwang et
al.~\cite{Hwang2019} and Guo et al.~\cite{GZZ2019} independently.

In the final section, Section~\ref{sec:laurent}, we consider the negative
version of the number of bounded Schr\"oder paths and the negative moments of
Laurent biorthogonal polynomials.

\section{Preliminaries}
\label{sec:preliminaries}

In this section we give some definitions related to negative moments of orthogonal
polynomials and prove their basic properties.

We say that a sequence \( (f_n)_{n\ge0} \) \emph{satisfies a homogeneous linear
  recurrence relation} if there exist a positive integer \( k \) and constants
\( r_1,\dots,r_k \) with \( r_k\ne0 \) such that for all \( n\ge k \),
\begin{equation}\label{eq:lrr}
  f_n = r_1f_{n-1}+\cdots+r_k f_{n-k}.
\end{equation}
In this case we can uniquely extend the sequence \( f_n \) to all integers \( n
\) by requiring that \eqref{eq:lrr} holds for all \( n\in\ZZ \). Therefore,
whenever a sequence \( (f_n)_{n\ge0} \) satisfies a homogeneous linear
recurrence relation, we can also consider the negatively indexed sequence \(
(f_{-n})_{n\ge1} \).

It is not hard to check that the ``negative of negative'' of a sequence is itself in the sense that
if we write \( f=(f_n)_{n\ge0} \) and \( \overline{f}=(f_{-n})_{n\ge0} \),
then  \( \overline{\overline{f}}=f \).

The following well-known lemma is useful when we study negatively indexed
sequences.

\begin{lem}\label{lem:f(-n)}\cite[Theorem~4.1.1]{ec1}
  A sequence \( (f_n)_{n\ge0} \) satisfies a homogeneous linear recurrence relation
  if and only if 
  \[
    \sum_{n\ge0} f_n x^n = \frac{P(x)}{Q(x)}
  \]
  for some polynomials \( P(x) \) and \( Q(x) \) with \( \deg(P(x))<\deg(Q(x))
  \) and \( Q(0)\ne0 \). Moreover, in this case, we have
  \[
    \sum_{n\ge1} f_{-n} x^n = -\frac{P(1/x)}{Q(1/x)},
  \]
  as rational functions.
\end{lem}

In this paper, a \emph{lattice path} is a finite sequence \( p=(p_0,p_1,\dots,p_n) \) of points
in \( \ZZ\times \ZZ_{\ge0}\). Each \( S_i=(x_{i}-x_{i-1}, y_{i}-y_{i-1}) \), \(
1\le i\le n \), is called a \emph{step} of \( p \). If the starting point \( p_0
\) is fixed, we will often identify the lattice path \( p \) with the sequence
\( (S_1,S_2,\dots,S_n) \) of its steps.

A \emph{Motzkin path} is a lattice path in which every step is an \emph{up step}
$U=(1,1)$, a \emph{horizontal step} $H=(1,0)$, or a \emph{down step} \( D=(1,-1)
\). We denote by \( \Mot_{n,r,s} \) the set of Motzkin paths from \( (0,r) \) to
\( (n,s) \). Let \( \Mot^{\le k}_{n,r,s} \) be the set of Motzkin paths in \(
\Mot_{n,r,s} \) that lie weakly below the line \( y=k \). We also define \(
\Mot_{n} = \Mot_{n,0,0} \) and \( \Mot^{\le k}_{n} = \Mot^{\le k}_{n,0,0} \).
  
  Throughout this paper we use the following notation:
\begin{align*}
  \vb&= (b_n)_{n\ge0}=(b_0,b_1,\dots),   \\
  \vla&= (\lambda_n)_{n\ge1}=(\lambda_1,\lambda_2,\dots),   \\
  \vb^2&= (b_{n-1}b_n)_{n\ge1} = (b_0b_1,b_1b_2,\dots), \\
  \vz&= (0,0,\dots), \\
  \vo&= (1,1,\dots).
\end{align*}

\begin{defn}
  The \emph{weight} \( \wt(\pi;\vb,\vla) \) of a Motzkin path \( \pi \) (with
  respect to \( \vb \) and \( \vla \)) is defined to be the product of \( b_i \)
  for each horizontal step starting at a point with \( y \)-coordinate \( i \)
  and \( \lambda_i \) for each down step starting at a point with \( y
  \)-coordinate \( i \). We define
  \begin{align*}
    \mu_{n,r,s}(\vb,\vla) &= \sum_{\pi\in \Mot_{n,r,s}} \wt(\pi;\vb,\vla),\\
    \mu_{n,r,s}^{\le k}(\vb,\vla) &= \sum_{\pi\in \Mot^{\le k}_{n,r,s}} \wt(\pi;\vb,\vla),\\
    \mu_{n}(\vb,\vla) &= \mu_{n,0,0}(\vb,\vla),\\
    \mu_{n}^{\le k}(\vb,\vla) &= \mu_{n,0,0}^{\le k}(\vb,\vla).
  \end{align*}
\end{defn}

Recall that \( P_n(x;\vb,\vla) \), \( n\ge0 \), are the orthogonal polynomials
defined by the three-term recurrence in \eqref{eq:OP}.

\begin{defn}\label{defn:inverted}
  The \emph{inverted polynomial} of \( P_n(x;\vb,\vla) \) is defined by \(
  P^*_n(x;\vb,\vla) = x^n P_n(1/x;\vb,\vla) \). We also define
  \[
    \delta P_n(x;\vb,\vla) = P_n(x;\delta \vb,\delta \vla),
  \]
  \[
    \delta P^*_n(x;\vb,\vla) = P^*_n(x;\delta \vb,\delta \vla),
  \]
  where for a sequence $\bm{s}=(s_n)_{n\ge 0}$ we denote $\delta \bm{s} =
  (s_{n+1})_{n\ge 0}$.
\end{defn}

The main focus of this paper is to study the negative versions of \(
\mu_{n,r,s}^{\le k}(\vb,\vla) \).

\begin{defn}
  Let \( k,r,s \) be fixed integers. If the sequence \( \mu_{n,r,s}^{\le
    k}(\vb,\vla) \) for \( n=0,1,\dots \) satisfies a homogeneous linear
  recurrence relation, then we define \( \mu_{-n,r,s}^{\le k}(\vb,\vla) \) for
  \( n=1,2,\dots \) in the unique way so that the sequence \( \mu_{n,r,s}^{\le
    k}(\vb,\vla) \) for all \( n\in\ZZ \) satisfies the recurrence. We call \(
  \mu_{-n}^{\le k}(\vb,\vla) := \mu_{-n,0,0}^{\le k}(\vb,\vla) \) the
  \emph{negative moments} of the orthogonal polynomials \( P_n(x;\vb,\vla) \).
\end{defn}

Now we prove some basic properties of the (generalized) negative moments \(
\mu_{-n,r,s}^{\le k}(\vb,\vla) \).

Viennot \cite[Ch.~V, (27)]{ViennotLN} found the following generating function
for \( \mu_{n,r,s}^{\le k}(\vb,\vla) \).

\begin{lem}\label{lem:P-ratio}
  Let \( r,s, \) and \( k \) be integers with \( 0\le r,s\le k \).
  If $r\le s$, then 
  \begin{equation}
    \label{eq:trunc_r<s2}
    \sum_{n\ge0} \mu_{n,r,s}^{\le k}(\vb,\vla) x^n
    = \frac{x^{s-r}P^*_r(x;\vb,\vla) \delta^{s+1} P^*_{k-s}(x;\vb,\vla)}{P^*_{k+1}(x;\vb,\vla)} .  
  \end{equation}
  If $r> s$, then
  \begin{equation}
    \label{eq:trunc_r>s2}
    \sum_{n\ge0} \mu_{n,r,s}^{\le k}(\vb,\vla) x^n
    = \frac{ P^*_s(x;\vb,\vla) \delta^{r+1} P^*_{k-r}(x;\vb,\vla)}{P^*_{k+1}(x;\vb,\vla)}
     \prod_{i=s+1}^{r} \lambda_i.
  \end{equation}
\end{lem}

By Lemmas~\ref{lem:f(-n)} and \ref{lem:P-ratio} we can also find the generating
function for \( \mu_{-n,r,s}^{\le k}(\vb,\vla) \).

\begin{prop}\label{prop:P-ratio2}
  Let \( r,s, \) and \( k \) be integers with \( 0\le r,s\le k \). Suppose that
  \( \mu_{-n,r,s}^{\le k}(\vb,\vla) \) is well defined for \( n\ge1 \). 
  If $r\le s$, then 
  \begin{equation}
    \label{eq:trunc_r<s3}
    \sum_{n\ge1} \mu_{-n,r,s}^{\le k}(\vb,\vla) x^n
    = -\frac{xP_r(x;\vb,\vla) \delta^{s+1} P_{k-s}(x;\vb,\vla)}{P_{k+1}(x;\vb,\vla)} .  
  \end{equation}
  If $r> s$, then
  \begin{equation}
    \label{eq:trunc_r>s3}
    \sum_{n\ge1} \mu_{-n,r,s}^{\le k}(\vb,\vla) x^n
    =- \frac{ x^{r-s+1}P_s(x;\vb,\vla) \delta^{r+1} P_{k-r}(x;\vb,\vla)}{P_{k+1}(x;\vb,\vla)}
     \prod_{i=s+1}^{r} \lambda_i.
  \end{equation}
\end{prop}

Using Flajolet's combinatorial theory of continued fractions
\cite{Flajolet1980}, Viennot \cite{ViennotLN} showed that
\begin{equation}\label{eq:Viennot}
  \sum_{n\ge0} \mu_n^{\le k}(\vb,\vla)x^n = 
  \cfrac{1}{
    1-b_0x- \cfrac{\lambda_1x^2}{
      1-b_1x- \cfrac{\lambda_2x^2}{
        1-b_2x- \genfrac{}{}{0pt}{1}{}{\displaystyle\ddots -
          \cfrac{\lambda_kx^2}{1-b_kx}}} }}.
\end{equation}

There is a similar continued fraction expression for the generating function for
\( \mu_{-n}^{\le k}(\vb,\vla) \).

\begin{prop}\label{prop:mu=cont}
  If \( (\mu_{-n}^{\le k}(\vb,\vla))_{n\ge1} \) is defined, we have
  \[
    \sum_{n\ge1} \mu_{-n}^{\le k}(\vb,\vla)x^n 
    =\cfrac{-x}{
      x-b_0- \cfrac{\lambda_1}{
        x-b_1- \cfrac{\lambda_2}{
          x-b_2- \genfrac{}{}{0pt}{1}{}{\displaystyle\ddots -
            \cfrac{\lambda_k}{x-b_k}}} }}.
  \]
\end{prop}
\begin{proof}
  By Lemma~\ref{lem:f(-n)} and \eqref{eq:Viennot},   
\[
  \sum_{n\ge1} \mu_{-n}^{\le k}(\vb,\vla)x^n 
 =\cfrac{-1}{
    1-b_0x^{-1}- \cfrac{\lambda_1x^{-2}}{
      1-b_1x^{-1}- \cfrac{\lambda_2x^{-2}}{
        1-b_2x^{-1}- \genfrac{}{}{0pt}{1}{}{\displaystyle\ddots -
          \cfrac{\lambda_kx^{-2}}{1-b_kx^{-1}}}} }}.
\]
Multiplying \( x \) to the numerator and the denominator for each fraction,
we obtain the desired formula.
\end{proof}

For the rest of this paper we mainly consider the bounded moments \(
\mu_{n,r,s}^{\le k}(\vb,\vla) \) and their negatives \( \mu_{-n,r,s}^{\le
  k}(\vb,\vla) \) when \( \vb=\vz \) or \( \vla=\vb^2 \). The choice of \( \vla
\) satisfying \( \vla=\vb^2 \) becomes more natural if we define the weight of a
Motzkin path using ``points'' instead of ``steps'' as follows: the
\emph{point-weight} \( \pwt(\pi;\vb) \) of a Motzkin path \( \pi\in \Mot_{n,r,s}
\) is defined by
  \[
    \pwt(\pi;\vb) = \prod_{(i,j)\in \pi} b_j.
  \]

If \( \vla=\vb^2 \), there is a simple relation between the usual weight \(
\wt(\pi;\vb,\vb^2) \) and the point-weight \( \pwt(\pi;\vb) \).

\begin{lem}\label{lem:vwt}
  For \( \pi\in \Mot_{n,r,s} \), we have
  \[
    \wt(\pi;\vb,\vb^2) = \frac{b_0\cdots b_{r-1}}{b_0\cdots b_{s}}\pwt(\pi;\vb).
  \]
\end{lem}
\begin{proof}
  We first show this for \( r=s=0 \). Suppose \( \tau\in \Mot_{n,0,0} \). Since
  each down step of \( \tau \) corresponds to a unique up step, we can
  redistribute the weight \( b_{i-1}b_i \) attached to a down step ending at
  height \( i-1 \) in such a way that the weight of the down step is \( b_{i-1}
  \) and the weight of its corresponding up step ending at height \( i \) is \(
  b_{i} \). Therefore \( \wt(\tau;\vb,\vb^2) \) is equal to the product of the
  new weights of the steps, where the weight of each step ending at height \( i
  \) is given by \( b_i \). This is equivalent to assigning the weight \( b_j \)
  for each lattice point \( (i,j) \) in \( \tau \) except the starting point \(
  (0,0) \). Thus \( \wt(\tau;\vb,\vb^2) = b_0^{-1}\pwt(\tau;\vb) \), which shows
  the lemma for \( r=s=0 \).

  Now consider the general case \( \pi\in \Mot_{n,r,s} \). Let \( \tau \) be the
  Motzkin path obtained from \( \pi \) by adding \( r \) up steps at the
  beginning and \( s \) down steps at the end. Then
  \[
    \wt(\pi;\vb,\vb^2) 
    =\frac{\wt(\tau;\vb)}{b_0b_1^2\cdots b_{s-1}^2 b_s},\qquad
    \pwt(\pi;\vb) 
    =\frac{\pwt(\tau;\vb)}{b_0\cdots b_{r-1} b_0 \cdots b_{s-1}} .
  \]
  Since \( \tau\in\Mot_{n+r+s,0,0} \), we have \( \wt(\tau;\vb,\vb^2) =
  b_0^{-1}\pwt(\tau;\vb) \), which together with the equations above implies the
  desired identity.
\end{proof}

Lemma~\ref{lem:vwt} immediately implies the following proposition, which shows
that \( \mu_{n,r,s}^{\le k}(\vb,\vb^2) \) is a natural point-weight generating
function for Motzkin paths.

\begin{prop}\label{prop:point_wt}
  We have
  \[
    \mu_{n,r,s}^{\le k}(\vb,\vb^2) =
    \frac{b_0\cdots b_{r-1}}{b_0\cdots b_{s}}
    \sum_{\pi\in \Mot^{\le k}_{n,r,s}} \pwt(\pi;\vb).
  \]
\end{prop}

We finish this section by giving sufficient conditions for \( \mu_{-n,r,s}^{\le
  k}(\vb,\vla) \) to be well defined.

\begin{prop}\label{prop:well-defined}
  If \( P_{k+1}(0;\vb,\vla)\ne0 \), then 
  \( \mu_{-n,r,s}^{\le k}(\vb,\vla) \) is well defined for \( n\ge1 \).
\end{prop}
\begin{proof}
  Since \( P_{k+1}(x;\vb,\vla) \) has the nonzero constant term \(
  P_{k+1}(0;\vb,\vla) \), its inverted polynomial \( P^*_{k+1}(x;\vb,\vla) \)
  has degree \( k+1 \). Moreover, \( P^*_{k+1}(0;\vb,\vla)=1 \) because it is
  the leading coefficient of the monic polynomial \( P_{k+1}(x;\vb,\vla) \).

Now we consider the generating
  function for \( \mu_{n,r,s}^{\le k}(\vb,\vla) \) in Lemma~\ref{lem:P-ratio}.
  If \( r\le s \),
  \[
    \deg(x^{s-r}P^*_r(x;\vb,\vla) \delta^{s+1} P^*_{k-s}(x;\vb,\vla))
    \le k < \deg(P^*_{k+1}(x;\vb,\vla)).
  \]
  If \( r>s \),
  \[
    \deg(P^*_s(x;\vb,\vla) \delta^{r+1} P^*_{k-r}(x;\vb,\vla))
    \le s+k-r< k < \deg(P^*_{k+1}(x;\vb,\vla)).
  \] 
  Therefore, by Lemma~\ref{lem:f(-n)}, \( \mu_{-n,r,s}^{\le k}(\vb,\vla) \) is
  well defined in either case.
\end{proof}

\begin{prop}\label{prop:well-defined2}
  The sequence \( (\mu_{-n,r,s}^{\le k}(\vz,\vla))_{n\ge1} \) is well defined if
  and only if \( k \) is odd. The sequence \( (\mu_{-n,r,s}^{\le
    k}(\vb,\vb^2))_{n\ge1} \) is well defined if and only if \( k\not\equiv 1
  \pmod 3 \).
\end{prop}
\begin{proof}
  Substituting \( x=0 \) in \eqref{eq:OP} gives a recurrence for \(
  P_{n}(0;\vb,\vla) \). Therefore, by induction, one can easily show that
  \begin{align*}
    P_{2k}(0;\vz,\vla) &= (-1)^{k} \prod_{i=1}^{k} \lambda_{2i-1},\\
    P_{2k+1}(0;\vz,\vla) &= 0,\\
    P_{3k}(0;\vb,\vb^2) &= b_0\cdots b_{3k-1},\\
    P_{3k+1}(0;\vb,\vb^2) &= -b_0\cdots b_{3k},\\
    P_{3k+2}(0;\vb,\vb^2) &= 0.
  \end{align*}
  Then the proof follows from Proposition~\ref{prop:well-defined}.
\end{proof}

\section{Reciprocity for bounded Dyck paths}
\label{sec:b=0}

In this section, we introduce a method to compute negative moments using
continued fractions. Using this method we give a combinatorial model for \(
\mu_{-n}^{\le k} (\vz,\vla) \), which is equivalent to Cigler and
Krattenthaler's result stated in Theorem~\ref{thm:CK}.

We begin with the following definitions.

\begin{defn}\label{def:PV}
  An \emph{\( \ell \)-peak-valley sequence} (\emph{\( \ell \)-PV sequence} for
  short) is a sequence \( (a_1,\dots,a_n) \) of nonnegative integers such that
  for \( i = 1,\dots,n \),
  \begin{itemize}
  \item if \( a_i\equiv 0 \pmod \ell\), then \( a_i \) is a valley, that is, \( a_{i-1}>a_i<a_{i+1} \),
  \item if \( a_i\equiv -1 \pmod \ell\), then \( a_i \) is a peak, that is, \( a_{i-1} < a_i > a_{i+1} \),
  \end{itemize}
  where we set \( a_0 = a_{n+1} = 0 \). Let \( \PV^{\ell,k}_{n} \) denote the
  set of all \( \ell \)-PV sequences \( (a_1,\dots,a_n) \) with bound \( k \),
  i.e., \( 0\le a_i\le k \) for all \( i=1,\dots,n \). 
\end{defn}

\begin{defn}
  We define the \emph{weight} \( \wt(\pi) \) of a sequence
  \( \pi = (a_1,\dots,a_n) \) of nonnegative integers by \( \wt(\pi) =
  V_{a_1}\cdots V_{a_n} \), where \( V_i \)'s are indeterminates.
\end{defn}

For convenience we define \( \PV^{\ell,k}_{0}=\{\emptyset\} \), where 
\( \emptyset \) is the empty sequence with \( \wt(\emptyset)=1 \).

In this paper we only need to consider \( \ell \)-PV sequences for \( \ell=2,3
\). In this section (resp.~Section~\ref{sec:lambda=bb}) we will show that if \(
\vb=\vz \) (resp.~\( \vla = \vb^2 \)), the negative moment \( \mu_{-n}^{\le k}
(\vb,\vla) \) is a generating function for certain \( 2 \)-PV sequences
(resp.~\( 3 \)-PV sequences).

Note that a sequence is a \( 2 \)-PV sequence if and only if every even integer
is a valley and every odd integer is a peak assuming a zero is padded at both
ends. For example, \( (3,2,7,0,1) \) is a \( 2 \)-PV sequence because the odd
integers \( 3,7,1 \) are peaks and the even integers \( 2,0 \) are valleys.
Equivalently, a sequence \( (a_1,\dots,a_n) \) is a \( 2 \)-PV sequence if and
only if \( n \) is odd, \( a_1>a_2<a_3>\cdots \), each \( a_{2i-1} \) is
odd, and \( a_{2i} \) is even.

Recall that \( \Alt_{n}^{\le k} \) is the set of all sequences \(
(b_1,\dots,b_n) \) such that \( b_1\le b_2 \ge b_3 \le \cdots \), and \( 1 \le
b_i \le k \). There is a close connection between alternating sequences and \( 2
\)-PV sequences as follows. 

\begin{prop}\label{prop:PV=Alt}
  The map from \( \PV_{2n+1}^{2,2k-1}\) to \( \Alt_{2n+1}^{\le k} \) defined by 
  \[
    (a_1,\dots,a_{2n+1})\mapsto \left(k-\flr{ a_1/2 },\dots,k-\flr{
        a_{2n+1}/2 } \right)
  \]
 is a bijection.
\end{prop}
\begin{proof}
  One can easily see that the map from \( \Alt_{2n+1}^{\le k} \) to \(
  \PV_{2n+1}^{2,2k-1}\) defined by \( (b_1,\dots,b_{2n+1})\mapsto
  (c_1,\dots,c_{2n+1}) \), where \( c_{2i+1}=2(k-b_{2i+1})+1 \) and \(
  c_{2i}=2(k-b_{2i}) \), is the inverse of the given map.
\end{proof}

There is a simple continued fraction expression for the generating function for
\( 2 \)-PV sequences.

\begin{prop}\label{prop:ContFrac=2_alt}
  For an integer \( k\ge 1 \), we have
  \[
    \sum_{n\ge 0}\sum_{\pi\in \PV^{2, 2k-1}_{2n+1}} \wt(\pi) x^{2n+1}
    =    \cfrac{1}{
      -V_0x- \cfrac{1}{
        -V_1x- \genfrac{}{}{0pt}{1}{}{\displaystyle\ddots -
          \cfrac{1}{-V_{2k-1}x}} }} .
  \]
\end{prop}
\begin{proof}
  The proof is by induction on \( k \). If \( k=1 \), since there is only
  one \( 2 \)-PV sequence \( (1,0,1,\dots,0,1) \) of length \( 2n+1 \) with
  bound \( 1 \), we have
  \begin{align*}
    \sum_{n\ge 0}\sum_{\pi\in \PV^{2, 1}_{2n+1}} \wt(\pi) x^{2n+1} 
    = \sum_{n\ge 0}V_1 x (V_1 V_0 x^2)^n
    = \frac{V_1 x}{1- V_1 V_0 x^2}.
  \end{align*}
  Hence it is true for \( k=1 \).
  
  Suppose \( k>1 \) and let \( \overline{\PV}^{2, 2k-1}_{2n+1} \) be the set of
  sequences \( (a_1,\dots,a_{2n+1}) \) in \( \PV^{2, 2k-1}_{2n+1} \) such that
  \( a_i \ge 2 \) for all \( i=1,\dots,2n+1 \). For convenience, let
  \begin{align*}
    \mathcal{A} = \bigcup_{n\ge 0} \PV^{2, 2k-1}_{2n+1} \qand
    \overline{\mathcal{A}} = \bigcup_{n\ge 0} \overline{\PV}^{2, 2k-1}_{2n+1}.
  \end{align*}
  Then, by induction with indices shifted properly, it is enough to show that
  \begin{align} \label{eq:ContFrac=2_alt}
    \sum_{\pi \in \mathcal{A}} \wt(\pi) x^{|\pi|} 
    = \frac{1}{-V_0x-\cfrac{1}{-V_1x-\sum_{\pi \in \overline{\mathcal{A}}} \wt(\pi) x^{|\pi|}}},
  \end{align}
  where if \( \pi=(a_1,\dots,a_n) \) we denote \( |\pi|=n \). 

  Let \( A=(a_1,\dots,a_{2n+1})\in \mathcal{A} \). We divide \( A \) into
  subsequences using the locations of \( 0 \)'s as follows.
  Let \( i_1,\dots,i_m
  \) be the indices \( j \) such that \( a_j = 0 \), where \( i_1<\dots<i_m \). 
  Let \( A_0=(a_1,\dots,a_{i_1-1}) \) and \(
  A_j=(a_{i_j},\dots,a_{i_{j+1}-1}) \) for \( j=1,\dots,m \), where \(
  i_{m+1}-1=2n+1 \), so that \( A \) is the concatenation of \( A_0,A_1,\dots,A_m
  \). For example, if \( A=(3,2,7,0,1,0,5,2,3,0,7,6,7) \), then \( A_0=(3,2,7)
  \), \( A_1=(0,1) \), \( A_2=(0,5,2,3) \), and \( A_3=(0,7,6,7) \).

  Since every even integer is a valley and every odd integer is a peak in \( A \),
  one can easily check the following:
  \begin{itemize}
  \item \( A_0 \) is either \( (1) \) or an element in \( \overline{\mathcal{A}}
    \), 
  \item for each \( 1\le j\le m \), \( A_j \) is either \( (0,1) \) or \(
    (0,y) \) for some \( y\in \overline{\mathcal{A}} \).
  \end{itemize}
  Conversely, any choice of \( A_0,A_1,\dots,A_m \) satisfying the above
  conditions gives an element in \( \mathcal{A} \).
  Therefore, if we set \( S = \sum_{\pi \in \overline{\mathcal{A}}} \wt(\pi)
  x^{|\pi|} \), then
  \begin{align*}
    \sum_{\pi \in \mathcal{A}} \wt(\pi) x^{|\pi|}  
    = \sum_{m \ge 0} (V_1 x + S ) (V_0V_1 x^2 + V_0 x S)^m
    = \frac{V_1x+S}{1-V_0V_1x^2-V_0xS}.
  \end{align*}
  Dividing the numerator and the denominator by \( V_1 x + S \), we obtain
  \eqref{eq:ContFrac=2_alt}, and the proof follows by induction.
\end{proof}

We now give a combinatorial interpretation for \( \mu_{-2n}^{\le 2k-1}
(\vz,\vla) \) using \( 2 \)-PV sequences.

\begin{thm}\label{thm:neg_mom b=0}
  Suppose that the sequence \( \vla=(\lambda_i)_{i\ge1} \) is given by \(
  \lambda_i = V_{i-1}^{-1} V_{i}^{-1} \) for \( i \ge 1 \). Then we have
  \begin{align*}
    \mu_{-2n}^{\le 2k-1} (\vz,\vla) = V_0  \sum_{\pi\in \PV^{2,2k-1}_{2n-1} } \wt(\pi).
  \end{align*}
\end{thm}
\begin{proof} 
  Let \( \lambda_0 = V_0^{-1} \). Observe that for each integer \( m\ge 0 \), \(
  \lambda_m^{-1}\lambda_{m-1}\cdots\lambda_0^{(-1)^{m+1}} = V_{m} \).
  By Proposition~\ref{prop:mu=cont},
  \begin{align*}
    \sum_{n\ge1} \mu_{-n}^{\le 2k-1} (\vz,\vla)x^n
    &=\cfrac{-x}{
      x- \cfrac{\lambda_1}{
        x- \cfrac{\lambda_2}{
          x- \genfrac{}{}{0pt}{1}{}{\displaystyle\ddots -
            \cfrac{\lambda_{2k-1}}{x}}} }}\\
      &=\cfrac{\lambda_0^{-1} x}{
        -\lambda_0^{-1} x- \cfrac{1}{
          -\lambda_1^{-1}\lambda_0 x- \cfrac{1}{
            -\lambda_2^{-1}\lambda_1 \lambda_0^{-1} x- \genfrac{}{}{0pt}{1}{}{\displaystyle\ddots -
        \cfrac{1}{-\lambda_{2k-1}^{-1}\lambda_{2k-2}\cdots \lambda_0^{(-1)^{2k}} x}}} }}\\
      &=\cfrac{V_0 x}{
        -V_0 x- \cfrac{1}{
          -V_1 x- \genfrac{}{}{0pt}{1}{}{\displaystyle\ddots -
        \cfrac{1}{-V_{2k-1} x}} }}.
  \end{align*}
  Then the proof follows from Proposition~\ref{prop:ContFrac=2_alt}.
\end{proof}

By the above theorem with \( \vla = \vo \), we get the following corollary.
\begin{cor}\label{cor:neg=PV}
  We have
  \begin{align*}
    \mu_{-2n}^{\le 2k-1} (\vz,\vo) = \left|\PV^{2,2k-1}_{2n-1}\right|.
  \end{align*}
\end{cor}

By Proposition~\ref{prop:PV=Alt}, Corollary~\ref{cor:neg=PV} is equivalent to
Theorem~\ref{thm:CK} due to Cigler and Krattenthaler. Moreover, using
Proposition~\ref{prop:PV=Alt} one can easily check that Theorem~\ref{thm:neg_mom
  b=0} is equivalent to the following proposition, which is a weighted version
of Theorem~\ref{thm:CK}.

\begin{prop}\label{prop:C-K result}\cite[Corollary~32]{Kratt_Hankel1}
  Suppose that \( \lambda_{2i-1} = V_i^{-1} A_i^{-1}  \) and \( \lambda_{2i} = A_i^{-1} V_{i+1}^{-1} \) for all \( i \ge 1 \), and let \( \lambda_0 = V_1^{-1} \). Then we have
  \begin{align*}
    \mu_{-2n}^{\le 2k-1} (\vz,\vla) = V_1 R_{AV}^{(k)} \left( \sum_{\pi \in \Alt_{2n-1}^{\le k}} \wt_{AV}(\pi) \right) ,
  \end{align*}
  where the operator \( R_{AV}^{(k)} \) replaces \( A_i \) by \( V_{k+1-i} \)
  and \( V_i \) by \( A_{k+1-i} \), and
  \begin{equation}\label{eq:wtAV}
    \wt_{AV}(\pi) = V_{a_1}V_{a_3}\cdots V_{a_{2n-1}} A_{a_2}A_{a_4}\cdots A_{a_{2n-2}} .
    \end{equation}
\end{prop}

\section{Reciprocity for bounded Motzkin paths}
\label{sec:lambda=bb}

In this section, we find a combinatorial interpretation for \( \mu_{-n}^{\le
  k}(\vb,\vb^2) \). We only need to consider the case \( k\not\equiv 1 \pmod 3
\) because otherwise \( \mu_{-n}^{\le k}(\vb,\vb^2) \) is not defined by
Proposition~\ref{prop:well-defined2}. We show that \( \mu_{-n}^{\le
  3k-1}(\vb,\vb^2) \) is a generating function for \( 3 \)-PV sequences
(Theorem~\ref{thm:neg3k-1}) and \( \mu_{-n}^{\le 3k}(\vb,\vb^2) \) is a
generating function for modified \( 3 \)-PV sequences (Theorem~\ref{thm:neg3k}).

Recall that a sequence \( (a_1,\dots,a_n) \) is a \( 3 \)-PV sequence if each \(
a_i \) is a valley when \( a_i \equiv 0 \pmod 3 \), and \( a_i \) is a peak when
\( a_i \equiv 2 \pmod 3 \).

Using arguments similar to those in the previous section we find a continued fraction
expression for the generating function for \( 3 \)-PV sequences.

\begin{prop}\label{prop:ContFrac=3_alt}
  For an integer \( k\ge 1 \), we have
  \[
    \sum_{n\ge 1}\sum_{\pi\in \PV^{3, 3k-1}_{n}} \wt(\pi) x^{n}= \cfrac{1}{
      -V_0x-1- \cfrac{1}{
        -V_1x-1- \genfrac{}{}{0pt}{1}{}{\displaystyle\ddots -
            \cfrac{1}{-V_{3k-1}x-1}} }}.
  \]
\end{prop}
\begin{proof} 
  The proof is similar to (but slightly more complicated than) that of
  Proposition~\ref{prop:ContFrac=2_alt}. Let
  \( \overline{\PV}^{3, 3k-1}_{n} \) be the set of sequences \(
  (a_1,\dots,a_{n}) \) in \( \PV^{3, 3k-1}_{n} \) such that \( a_i \ge 3 \) for
  all \( i=1,\dots,n \), and let
  \begin{align*}
    \mathcal{A} = \bigcup_{n\ge 0} \PV^{3, 3k-1}_{n} \qand
    \overline{\mathcal{A}} = \bigcup_{n\ge 0} \overline{\PV}^{3, 3k-1}_{n}.
  \end{align*}
  We first claim that 
  \begin{align}\label{eq:ContFrac=3_alt1}
    \sum_{\pi\in \mathcal{A}}\wt(\pi)x^{|\pi|}
    =\cfrac{1}{-V_0x-1-\cfrac{1}{-V_1x-1-\cfrac{1}{-V_2x-1-\sum_{\pi\in \overline{\mathcal{A}}}\wt(\pi)x^{|\pi|}}}}.
  \end{align}
   It is easy to see
  that the proposition follows from the claim by induction on \( k \).
  Therefore, it suffices to prove the claim~\eqref{eq:ContFrac=3_alt1}.
  
  For a sequence \( \pi=(a_1,\dots,a_n)\in\mathcal{A} \), let \( i_1,\dots,i_m
  \) be the indices \( j \) such that \( a_j = 0 \) or \(
  a_{j-1}\ge a_j=1 \), where \( i_1<\dots<i_m \). Let \(
  A_0=(a_1,\dots,a_{i_1-1}) \) and \( A_j=(a_{i_j},\dots,a_{i_{j+1}-1}) \) for
  \( j=1,\dots,m \), where \( i_{m+1}-1=n \), so that \( \pi \) is the
  concatenation of \( A_0,A_1,\dots,A_m \).

  Observe that the possible sequences for \( A_0 \) are \( (1),(1,2),(2),(1,y)
  \), and \( (y) \) where \( y \in \overline{\mathcal{A}} \). For \( 1\le j\le m
  \), the first entry of \( A_j \) is \( 0 \) or \( 1 \). If the first entry is
  \( 0 \), the possible sequences for \( A_j \) are \(
  (0,1),(0,1,2),(0,2),(0,1,y), (0,y) \), where \( y \in \overline{\mathcal{A}}
  \), and if the first entry is \( 1 \), the possible sequences for \( A_j \)
  are \( (1),(1,2), (1,y) \) where \( y \in \overline{\mathcal{A}} \). Hence, if
  we set \( S =\sum_{\pi\in \overline{\mathcal{A}}}\wt(\pi)x^{|\pi|} \), then we
  have
  \begin{align*}
    \sum_{\pi \in \mathcal{A}} \wt(\pi) x^{|\pi|}  
    &= \sum_{m \ge 0} (V_1x+V_1V_2x^2+V_2x+V_1x S+S ) \\
    &\qquad\qquad\qquad \times(V_0x(V_1x+V_1V_2x^2+V_2x+V_1x S+S)+V_1x(1+V_2x+S))^m\\
    &= \frac{V_1x+V_1V_2x^2+V_2x+V_1x S+S}
    {1-V_0x(V_1x+V_1V_2x^2+V_2x+V_1x S+S)-V_1x(1+V_2x+S)},
  \end{align*}
  which is easily seen to be equal to the right-hand side of
  \eqref{eq:ContFrac=3_alt1}. This completes the proof.
\end{proof}

Using Proposition~\ref{prop:ContFrac=3_alt} we can find a combinatorial interpretation
for \( \mu_{-n}^{\le 3k-1}(\vb,\vb^2) \).

\begin{thm}\label{thm:neg3k-1}
  Let \( \vb=(b_i)_{i\ge0} \) and \( \vla=(\lambda_i)_{i\ge1} \) be the
  sequences given by \( b_i=-V_i^{-1} \) and \(
  \lambda_{i}=V_{i}^{-1}V_{i-1}^{-1} \) for all \( i \). Then we have
  \begin{align*}
    \mu_{-n}^{\le 3k-1}(\vb,\vla) = V_0 \sum_{\pi\in \PV^{3, 3k-1}_{n-1}} \wt(\pi).
  \end{align*}
\end{thm}
\begin{proof} 
  By Proposition~\ref{prop:mu=cont}, we have
  \begin{align*}
    \sum_{n\ge1} \mu_{-n}^{\le k} (\vb,\vla)x^n
    &=\cfrac{-x}{
      x-b_0- \cfrac{\lambda_1}{
      x-b_1- \cfrac{\lambda_2}{
      x-b_2- \genfrac{}{}{0pt}{1}{}{\displaystyle\ddots -
      \cfrac{\lambda_k}{x-b_k}}} }}\\
    &=\cfrac{b_0^{-1} x}{
    1-b_0^{-1}x- \cfrac{b_0^{-1}b_1^{-1}\lambda_1}{
      1-b_1^{-1}x- \cfrac{b_1^{-1}b_2^{-1}\lambda_2}{
    1-b_2^{-1}x- \genfrac{}{}{0pt}{1}{}{\displaystyle\ddots -
      \cfrac{b_{k-1}^{-1}b_k^{-1}\lambda_k}{1-b_k^{-1}x}}} }}\\
    &=\cfrac{V_0 x}{
      -V_0 x - 1- \cfrac{1}{
      -V_1x - 1 - \cfrac{1}{
      -V_2x - 1 - \genfrac{}{}{0pt}{1}{}{\displaystyle\ddots -
      \cfrac{1}{-V_k x -  1}}} }}.
  \end{align*} 
  The proof follows from Proposition~\ref{prop:ContFrac=3_alt}.
\end{proof}

Now we find a combinatorial interpretation for \( \mu_{-n}^{\le 3k}(\vb,\vb^2)
\). To this end we need the following definition.

\begin{defn} 
  A \emph{modified \( 3 \)-PV sequence} is a sequence \( (a_1,\dots,a_n) \) of
  nonnegative integers such that for \( i =1,\dots,n \),
 \begin{itemize}
 \item if \( a_i\equiv 1 \pmod 3\), then \( a_i \) is a valley, that is, \( a_{i-1}>a_i<a_{i+1} \),
 \item if \( a_i\equiv 2 \pmod 3\), then \( a_i \) is a peak, that is, \( a_{i-1} < a_i > a_{i+1} \),
  \end{itemize}
  where we set \( a_0 = a_{n+1} = 0 \). Let \( \widetilde{\PV}^{3,k}_{n} \)
  denote the set of all modified \( 3 \)-PV sequences of length \( n \) with
  bound \( k \), i.e., \( 0\le a_i\le k \) for all \( i \).
\end{defn}

Similar to Proposition~\ref{prop:ContFrac=3_alt}, there is a continued fraction
expression for the generating function for \( \widetilde{\PV}_n^{3,3k} \), see the proposition below. We note, however, that the proof of
Proposition~\ref{prop:ContFrac=3_alt2} is different from that of
Proposition~\ref{prop:ContFrac=3_alt} due to the fact that in
Proposition~\ref{prop:ContFrac=3_alt2} the sum is over \( n\ge0 \) whereas in
Proposition~\ref{prop:ContFrac=3_alt} the sum is over \( n\ge1 \).

\begin{prop}\label{prop:ContFrac=3_alt2}
  For an integer \( k\ge 1 \), we have
  \[
    \sum_{n\ge 0} (-1)^{n+1}\sum_{\pi \in \widetilde{\PV}_n^{3,3k}} \wt(\pi) x^{n} =
    \cfrac{1}{
      -V_0x-1- \cfrac{1}{
        -V_1x-1- \genfrac{}{}{0pt}{1}{}{\displaystyle\ddots -
            \cfrac{1}{-V_{3k}x-1}} }}.
  \]
\end{prop}
\begin{proof}
  Let \( \mathcal{A}=\cup_{n\ge 0}\widetilde{\PV}_n^{3,3k} \) and let \(
  \mathcal{B} \) be the set of sequences \( \beta=(b_1,\dots,b_m) \) for \( m\ge
  0 \) such that \( b_i \) is a valley if \( b_i \equiv 1 \pmod 3 \), \( b_i \)
  is a peak if \( b_i \equiv 0 \pmod 3 \), and \( 1\le b_i \le 3k \) for all \(
  i \), where we set \( b_0 = b_{m+1} = 0 \).
  Here \( (b_1,\dots,b_m)\) means the empty sequence \(\emptyset \) if \(
    m=0 \). By Proposition~\ref{prop:ContFrac=3_alt}, we have
  \begin{align*}
    \sum_{\beta\in \mathcal{B}} \wt(\beta) x^{|\beta|} = 1 + \cfrac{1}{
      -V_1x-1- \cfrac{1}{
        -V_2x-1-  \genfrac{}{}{0pt}{1}{}{\displaystyle\ddots -
            \cfrac{1}{-V_{3k}x-1}}} }.
  \end{align*}

  We claim that
  \begin{align}\label{eq:claim3k}
    \sum_{\alpha\in \mathcal{A}}\wt(\alpha)(-x)^{|\alpha|} \left( V_0 x + \sum_{\beta\in \mathcal{B}}\wt(\beta)x^{|\beta|} \right) = 1.
  \end{align}
  For \( \alpha\in\mathcal{A} \) and \( \beta\in \mathcal{B}\cup \{(0)\} \),
  define the weight \( \overline{\wt}(\alpha,\beta) \) of the pair \(
  (\alpha,\beta) \) to be \( (-1)^{|\alpha|}\wt(\alpha)\wt(\beta) \). To prove
  the claim it suffices to find a sign-reversing involution \( \varphi \) from
  \( \mathcal{A}\times (\mathcal{B}\cup \{(0)\}) \) to itself with unique fixed
  point \( (\emptyset,\emptyset) \), where \( \emptyset \) is the empty
  sequence.

  For a nonempty sequence \( \alpha = (a_1,\dots,a_m)\in\mathcal{A} \), define
  \( I(\alpha)\) to be the largest integer \( i \) such that \( 1\le i\le m-1 \)
  and \( a_i, a_{i+1} \not\equiv 1 \pmod 3\). If there is no such \( i \), we
  define \( I(\alpha)=0 \). Similarly, for a nonempty sequence \( \beta =
  (b_1,\dots,b_n)\in\mathcal{B}\cup\{(0)\} \), define \( J(\beta) \) to be the
  smallest integer \( j \) such that \( 1\le j\le n-1 \) and \(
  b_j,b_{j+1}\not\equiv 1 \pmod 3 \). If there is no such \( j \), we define \(
  J(\beta)=n \). One can check that \( m-I(\alpha) \) and \( J(\beta) \) are
  odd. Moreover, \( a_{I(\alpha)+1}>a_{I(\alpha)+2}<\dots<a_m \) and \(
  b_1>b_2<\dots<b_{J(\beta)} \).

  We define the map \( \varphi \) as follows. For \( \alpha = (a_1,\dots,a_m)
  \in \mathcal{A} \) and \( \beta = (b_1,\dots,b_n)\in\mathcal{B}\cup \{(0)\} \),
  \begin{enumerate}
    \item define \( \varphi(\alpha,\beta) = ((a_1,\dots,a_m,b_1,\dots,b_{J(\beta)}),(b_{J(\beta)+1},\dots,b_n)) \) if one of the following conditions is satisfied:
    \[\begin{cases}
      m=0, \\
      a_m \equiv 0 \pmod 3 \qand  b_1 \equiv 0 \pmod 3, \\
      a_m\equiv 2\pmod 3  \qand  b_1 \equiv 0\pmod 3  \mbox{ with }  a_m>b_1,\\
      a_m\equiv 0\pmod 3  \qand b_1 \equiv 2\pmod 3  \mbox{ with } a_m<b_1,
  \end{cases}\]
  \item define \( \varphi(\alpha,\beta) = ((a_1,\dots,a_{I(\alpha)}),(a_{I(\alpha)+1},\dots,a_m ,b_1,\dots,b_n)) \) if one of the following conditions is satisfied:
  \[\begin{cases}
    n=0,\\
      a_m \equiv 2 \pmod 3 \qand  b_1 \equiv 2 \pmod 3, \\
      a_m\equiv 2\pmod 3  \qand  b_1 \equiv 0\pmod 3  \mbox{ with }  a_m<b_1, \\
      a_m\equiv 0\pmod 3  \qand b_1 \equiv 2\pmod 3  \mbox{ with } a_m>b_1.
  \end{cases}\] 
  \end{enumerate}
  Then it is not hard to see that the map \( \varphi \) is a sign-reversing
  involution with unique fixed point \( (\emptyset,\emptyset) \), which proves
  the claim.
  For example, let \( \alpha = (a_1,a_2,a_3,a_4,a_5)= (2,0,3,1,5)\in \mathcal{A} \) and \( \beta = (b_1,b_2,b_3,b_4) = (2,3,1,2)\in \mathcal{B} \). Then \( I(\alpha) = 2 \) since \( a_2,a_3\not\equiv 1 \pmod 3  \), and \( J(\beta)=1 \) since \( b_1,b_2\not\equiv 1 \pmod 3 \). Since \( a_5\equiv 2\pmod 3   \) and \( b_1 \equiv 2 \pmod 3 \), it satisfies the second condition of the case (2), so \( \varphi(\alpha,\beta) = ((2,0),(3,1,5,2,3,1,2)) \).
  Moreover, one can easily check that \( \varphi((2,0),(3,1,5,2,3,1,2)) = ((2,0,3,1,5),(2,3,1,2)) = (\alpha,\beta) \).

  By the claim~\eqref{eq:claim3k}, we have
  \begin{align*}
    -\sum_{\alpha\in \mathcal{A}}\wt(\alpha)(-x)^{|\alpha|}
    &= \frac{1}{-V_0 x - \sum_{\beta\in \mathcal{B}}\wt(\beta)x^{|\beta|}} \\
    &= \cfrac{1}{
        -V_0x-1- \cfrac{1}{
          -V_1x-1-  \genfrac{}{}{0pt}{1}{}{\displaystyle\ddots -
              \cfrac{1}{-V_{3k}x-1}}} },
  \end{align*}
  which completes the proof.
\end{proof}

Similar to Theorem~\ref{thm:neg3k-1}, using Proposition~\ref{prop:ContFrac=3_alt2} we
can find a combinatorial interpretation for \( \mu_{-n}^{\le 3k}(\vb,\vb^2)
\). We omit the proof.

\begin{thm} \label{thm:neg3k}
  Let \( \vb=(b_i)_{i\ge0} \) and \( \vla=(\lambda_i)_{i\ge1} \) be the
  sequences given by \( b_i=-V_i^{-1} \) and \(
  \lambda_{i}=V_{i}^{-1}V_{i-1}^{-1} \) for all \( i \). Then we have
  \begin{align*}
    \mu_{-n}^{\le 3k}(\vb,\vla) = V_0 \sum_{\pi \in \widetilde{\PV}_{n-1}^{3,3k}} \wt(\pi).
  \end{align*}
\end{thm}

\section{Negative moments using inverse matrices}
\label{sec:matrices}

In this section, we generalize Theorems~\ref{thm:neg3k-1} and \ref{thm:neg3k}
using inverse matrices.

For integers \( k \) and \( i \) with \( 0\le i\le k \), let \( \epsilon^{\le
  k}_i \) be the standard basis vector in \( \RR^{k+1} \) such that the \( i
\)th entry is equal to \( 1 \) and the other entries are all \( 0 \). If the
size of \( \epsilon_i^{\le k} \) is clear from the context we will simply write
it as \( \epsilon_i \). We also define the tridiagonal matrix \( A^{\le
  k}(\vb,\vla) \) by
\begin{equation}\label{eq:Ak(b,la)}
  A^{\le k}(\vb,\vla) = \begin{pmatrix}
    b_0  & 1 & & &  \\
    \lambda_1  & b_1 & 1& &  \\
    & &\ddots& &  \\
    & & \lambda_{k-1}& b_{k-1}& 1 \\
    && & \lambda_k& b_k \\
  \end{pmatrix}.
\end{equation}

By the definition of \( \mu_{n,r,s}^{\le k}(\vb,\vla) \),
it is easy to see that
\begin{equation}\label{eq:matrix_mu}
  \mu_{n,r,s}^{\le k}(\vb,\vla) = \epsilon_r^T \left( A^{\le k}(\vb,\vla) \right)^n \epsilon_s.
\end{equation}
The next proposition shows that \( \mu_{-n,r,s}^{\le k}(\vb,\vla) \)
can be computed similarly using the inverse of
\( A^{\le k}(\vb,\vla) \). This is essentially the same as
\cite[Lemma~2.7]{Hopkins2023} due to Hopkins and Zaimi, which was
first appeared in \cite{zaimi}.

\begin{prop} \cite[Lemma~2.7]{Hopkins2023}
  \label{prop: matrix interpretation of negative moment}
For
  nonnegative integers \( r,s,k,n \) with \( r,s\le k \) and \( n\ge1 \), if \(
  A^{\le k}(\vb,\vla) \) is invertible, then
\begin{equation}\label{eq:negative_matrix_mu}
  \mu_{-n,r,s}^{\le k}(\vb,\vla) = \epsilon_r^T \left( A^{\le k}(\vb,\vla) \right)^{-n} \epsilon_s.
\end{equation}
\end{prop}
\begin{proof}
    Let \( x^m + c_{m-1}x^{m-1}+\dots+c_0 \) be the minimal polynomial of \(
    A^{\le k}(\vb,\vla) \) so that 
    \[
      (A^{\le k}(\vb,\vla))^m + c_{m-1}(A^{\le k}(\vb,\vla))^{m-1}+\dots+c_0I =O,
    \]
    where \( I \) (resp.~\( O \)) is the identity matrix (resp.~zero matrix).
    For each \( N\in\ZZ \), multiplying \( (A^{\le k}(\vb,\vla))^{N-m} \) and
    then multiplying \( \epsilon_r^T \) and \( \epsilon_s \) on the left and
    right, respectively, in the above equation, we obtain
    \[
      \epsilon_r^T(A^{\le k}(\vb,\vla))^N\epsilon_s + c_{m-1}\epsilon_r^T(A^{\le
        k}(\vb,\vla))^{N-1}\epsilon_s+\dots+c_0\epsilon_r^T(A^{\le k}(\vb,\vla))^{N-m} \epsilon_s=0 .
    \]
    Therefore \( (\epsilon_r^T(A^{\le k}(\vb,\vla))^N\epsilon_s)_{N\in\ZZ} \) is the sequence
    that is extended from
    \( (\mu_{N,r,s}^{\le k}(\vb,\vla))_{N\ge0} = (\epsilon_r^T(A^{\le
      k}(\vb,\vla))^N\epsilon_s)_{N\ge0} \) by the above linear recurrence relation,
    which implies \eqref{eq:negative_matrix_mu}.
\end{proof}

Usmani \cite{usmani1994inversion} found a formula for the inverse of a general
tridiagonal matrix. Specializing Usmani's result to the tridiagonal matrix \(
A^{\le k}(\vb,\vla) \) we obtain the following lemma.

\begin{lem}\label{lem:usmani}
Suppose that \( A^{\le k}(\vb,\vla) \) is invertible and 
let \( \left( A^{\le k}(\vb,\vla) \right)^{-1} = (\alpha_{i,j})_{0\le i,j\le k} \).
Then
 \begin{align}\label{eq:usmani}
  \alpha_{i,j}=\begin{cases}
    (-1)^{i+j}\theta_{i}\phi_{j+2}/\theta_{k+1} & \mbox{if $i\le j$},\\ 
    (-1)^{i+j}\lambda_j\cdots\lambda_{i-1}\theta_{j}\phi_{i+2}/\theta_{k+1} & \mbox{if \( i>j \)},
  \end{cases}
 \end{align}
 where \( \theta_i \) and \( \phi_i \) are defined by
  \begin{align*}
    \theta_i&=b_{i-1}\theta_{i-1}-\lambda_{i-1}\theta_{i-2},&i&=1,2,\dots,k+1,\\
    \phi_i&=b_{i-1}\phi_{i+1}-\lambda_i\phi_{i+2},&i&=k+1,k,\dots,1,
  \end{align*}
  with initial conditions \( \phi_{k+2}=\theta_0=1\) and \(\phi_{k+3}=\theta_{-1}=0 \).
\end{lem}

The next lemma shows that if \( \vla=\vb^2 \), then there is a simple explicit
formula for \( \alpha_{i,j} \) in Lemma~\ref{lem:usmani}.

\begin{lem}\label{lem:vv-inv}
  Let \( \vb=(b_i)_{i\ge0} \) and \( \vla=(\lambda_i)_{i\ge1} \) be the
  sequences given by \( b_i=-V_i^{-1} \) and \(
  \lambda_{i}=V_{i}^{-1}V_{i-1}^{-1} \) for all \( i \). Suppose that \( A^{\le
    k}(\vb,\vla) \) is invertible and let \( \left( A^{\le k}(\vb,\vla)
  \right)^{-1} = (\alpha_{i,j})_{0\le i,j\le k} \). Then
  \begin{align*}
    \alpha_{i,j}=(-1)^{\flr{\frac{i}{3}}+\flr{\frac{j}{3}}}\frac{V_0\cdots V_j}{V_0\cdots V_{i-1}}\chi_{i,j},
  \end{align*}
  where for \( k\equiv-1\pmod3 \), 
  \[ \chi_{i,j}=\begin{cases}
    0 & \text{ if } \quad i\equiv-1\pmod3 \text{ and } i\le j,\\
    0 & \text{ if } \quad j\equiv0\pmod3 \text{ and } i\le j,\\
    0 & \text{ if } \quad i\equiv0\pmod3  \text{ and } i>j,\\
    0 & \text{ if } \quad j\equiv-1\pmod3 \text{ and } i>j,\\
    1 & \text{ otherwise,}
  \end{cases} 
  \]
  and for \( k\equiv0\pmod3 \), 
  \[ \chi_{i,j}=\begin{cases}
    0 & \text{ if } \quad i\equiv-1\pmod3 \text{ and } i\le j,\\
    0 & \text{ if } \quad j\equiv1\pmod3 \text{ and } i\le j,\\
    0 & \text{ if } \quad i\equiv1\pmod3  \text{ and } i>j,\\
    0 & \text{ if } \quad j\equiv-1\pmod3 \text{ and } i>j,\\
    1 & \text{ otherwise.}
  \end{cases} 
  \]
\end{lem}
\begin{proof}
  By induction on \( i \), one can easily verify that
  the \( \theta_i \)'s and \( \phi_i \)'s in Lemma~\ref{lem:usmani} are given by
  \begin{align*} 
    \theta_{3i}&=V_0^{-1}\cdots V_{3i-1}^{-1},& \phi_{k+1-3i}&=-V_{k}^{-1}\cdots V_{k-3i}^{-1},\\
    \theta_{3i+1}&=-V_0^{-1}\cdots V_{3i}^{-1},&\phi_{k+2-3i}&=V_{k}^{-1}\cdots V_{k+1-3i}^{-1},\\
    \theta_{3i+2}&=0,&\phi_{k+3-3i}&=0.
  \end{align*}
  If \( k\equiv-1\pmod3 \), then these can be written as 
  \begin{align*} 
    \theta_i&=\begin{cases}
     0 & \mbox{ if \( i\equiv-1\pmod 3 \)},\\ 
     (-1)^{i-3\flr{\frac{i}{3}}} V_{0}^{-1}\cdots V_{i-1}^{-1} & \mbox{ otherwise,}
    \end{cases}\\
    \phi_i&=\begin{cases}
      0 & \mbox{ if \( i\equiv-1\pmod 3 \)},\\ 
      (-1)^{i-3\flr{\frac{i}{3}}+1} V_{i-1}^{-1}\cdots V_{k}^{-1} & \mbox{ otherwise.}
     \end{cases}
  \end{align*}
  Substituting the formulas for \( \phi_i \) and \( \theta_i \) to
  \eqref{eq:usmani} completes the proof for \( k\equiv-1\pmod3 \). One can obtain the result for \( k\equiv0\pmod3 \) in the same way.
\end{proof}

Note that if \( k\equiv1\pmod3 \), then \( \alpha_{i,j} \) in \eqref{eq:usmani} is not defined, i.e., \( A^{\le
k}(\vb,\vla) \) is not invertible. Using Proposition~\ref{prop: matrix interpretation of negative moment} and
Lemma~\ref{lem:usmani} we can give a combinatorial interpretation for \(
\mu_{-n,r,s}^{\le k}(\vb,\vb^2) \). To do this, we first need to define \(
(\ell,r,s) \)-peak-valley sequences, which are a generalization of \( \ell
\)-peak-valley sequences in Definition~\ref{def:PV}.

\begin{defn}\label{def:PVnrs}
  An {\it \( (\ell,r,s) \)-peak-valley sequence} is a sequence \( (
    a_1,\dots,a_{n}) \) of nonnegative integers such that for \( i=0,\dots,n+1 \),
  \begin{itemize}
  \item if \( a_i\equiv 0\pmod\ell \), then \( a_i \) is a valley, that is, \( a_{i-1}>a_i<a_{i+1} \),
  \item if \( a_i\equiv -1\pmod\ell \), then \( a_i \) is a peak, that is, \( a_{i-1}<a_i>a_{i+1} \),
  \end{itemize}
  where we set \( a_0=r \) and \( a_{n+1} = s \). Here, we ignore the
  inequalities involving \( a_t \) for \( t=-1 \) or \( t=n+2 \). Denote by \(
  \PV_{n,r,s}^{\ell,k} \) the set of \( (\ell,r,s) \)-peak-valley sequences \( (
  a_1,\dots,a_{n}) \) with bound \( k \), i.e., \( 0\le a_i\le k \) for all \(
  i=1,\dots,n \).
\end{defn}

Note that Definition~\ref{def:PVnrs} reduces to Definition~\ref{def:PV} when \(
r=s=0 \). Cigler and Krattenthaler~\cite[Theorem 28]{Kratt_Hankel1} found a
combinatorial description of \( \mu_{-n,r,s}^{\le 2k-1}(\vz,\vla) \). The next
theorem gives a combinatorial interpretation for \( \mu_{-n,r,s}^{\le
  3k-1}(\vb,\vla) \) when \( b_i=-V_i^{-1} \) and \(
\lambda_{i}=V_{i}^{-1}V_{i-1}^{-1} \) for all \( i \). Note that this theorem
reduces to Theorem~\ref{thm:neg3k-1} if \( r=s=0 \).

\begin{thm}\label{thm:rsneg3k-1}
  Suppose that \( \vb=(b_i)_{i\ge0} \) and \( \vla=(\lambda_i)_{i\ge1} \) are
  the sequences given by \( b_i=-V_i^{-1} \) and \(
  \lambda_{i}=V_{i}^{-1}V_{i-1}^{-1} \) for all \( i \). Then
  \[ 
    \mu_{-n,r,s}^{\le 3k-1}(\vb,\vla) =
    (-1)^{\flr{{r}/{3}}+\flr{{s}/{3}}}
    \frac{V_0\cdots V_s}{V_0\cdots V_{r-1}}\sum_{\pi\in\PV_{n-1,r,s}^{3,3k-1}}\wt(\pi).
  \]
  Here, we set \( V_0\cdots V_{r-1}=1 \) if \( r=0 \).
\end{thm}
\begin{proof}
  Let \( a_0=r \), \( a_n=s \) and \( (\alpha_{i,j})_{0\le i,j\le
    3k-1}=\left(A^{\le 3k-1}(\vb,\vla)\right)^{-1} \). By Proposition~\ref{prop:
    matrix interpretation of negative moment},
  \begin{equation}\label{eq:mu-nrs} 
    \mu_{-n,r,s}^{\le 3k-1}(\vb,\vla)=\epsilon_r^T\left(A^{\le
        3k-1}(\vb,\vla)\right)^{-n}\epsilon_s=\sum_{(a_1,\dots,a_{n-1})\in X}\prod_{i=0}^{n-1} \alpha_{a_i,a_{i+1}},
  \end{equation}
  where \( X \) is the set of sequences \( (a_1,\dots,a_{n-1}) \) of integers
  with \(0\le a_i \le 3k-1 \) for all \( i \).

  For \( (a_1,\dots,a_{n-1})\in X \), we claim that \(
  \prod_{i=0}^{n-1}\alpha_{a_i,a_{i+1}}=0 \) unless \(
  (a_1,\dots,a_{n-1})\in\PV_{n-1,r,s}^{3,3k-1} \). To see this, suppose \(
  (a_1,\dots,a_{n-1})\not\in\PV_{n-1,r,s}^{3,3k-1} \). Then there is an integer
   \( 0\le j\le n \) satisfying one of the
  following two conditions:
  \begin{itemize}
  \item \( a_{j}\equiv 0\pmod3 \) and \( a_{j} \) is not a valley,
  \item \( a_{j}\equiv -1\pmod3 \) and \( a_{j} \) is a not peak.
  \end{itemize}
  First, suppose \( a_{j}\equiv 0\pmod3 \) and \( a_{j} \) is not a valley. Then
  \( a_{j-1}\le a_{j} \) or \( a_{j}\ge a_{j+1} \). By Lemma~\ref{lem:vv-inv},
  \( a_{j-1}\le a_{j} \) implies \(\alpha_{a_{j-1},a_{j}}=0 \) and each of \(
  a_{j}= a_{j+1} \) and \( a_{j}>a_{j+1} \) implies \(\alpha_{a_{j},a_{j+1}}=0
  \). Hence we always have \( \prod_{i=0}^{n-1} \alpha_{a_i,a_{i+1}}=0 \).
  Similarly, one can prove \( \prod_{i=0}^{n-1} \alpha_{a_i,a_{i+1}}=0 \) in the
  second case that \( a_{j}\equiv -1\pmod3 \) and \( a_{j} \) is not a peak for
  some integer \( j \).
  
  By \eqref{eq:mu-nrs} and the claim, we have
  \begin{equation}\label{eq:prod_alpha} 
  \mu_{-n,r,s}^{\le 3k-1}(\vb,\vla)=\sum_{(a_1,\dots,a_{n-1})\in\PV_{n-1,r,s}^{3,3k-1}}
  \prod_{i=0}^{n-1} \alpha_{a_i,a_{i+1}}.
  \end{equation}
  By Lemma~\ref{lem:vv-inv}, for \( \pi=(a_1,\dots,a_{n-1})\in\PV_{n-1,r,s}^{3,3k-1}
  \), we have
\begin{align*}
  \prod_{i=0}^{n-1} \alpha_{a_i,a_{i+1}}&=(-1)^{\flr{\frac{a_0}{3}}+2\flr{\frac{a_1}{3}}+\cdots+2\flr{\frac{a_{n-1}}{3}}+\flr{\frac{a_n}{3}}}\prod_{i=0}^{n-1}\frac{V_0\cdots V_{a_{i+1}}}{V_0\cdots V_{a_i}}V_{a_i}\\
  &=(-1)^{\flr{\frac{r}{3}}+\flr{\frac{s}{3}}}\frac{V_0\cdots V_s}{V_0\cdots V_{r-1}}V_{a_1}\cdots V_{a_{n-1}},
\end{align*}
which together with \eqref{eq:prod_alpha} gives the theorem.
\end{proof}

Putting \( V_i=-1 \) in Theorem~\ref{thm:rsneg3k-1} gives the following corollary.
\begin{cor} 
  We have
  \[ 
    \mu_{-n,r,s}^{\le 3k-1}(\vo,\vo)=(-1)^{\flr{{r}/{3}}+\flr{{s}/{3}}+r+s+n}\left|\PV_{n-1,r,s}^{3,3k-1}\right|.
  \]
\end{cor}

Similarly, we can find a combinatorial interpretation for \( \mu_{-n,r,s}^{\le
  3k}(\vb,\vla) \). To do this we introduce modified peak-valley sequences.
\begin{defn}
  A {\it modified \( (\ell,r,s) \)-peak-valley sequence} is a sequence \( (
    a_1,\dots,a_{n}) \) of nonnegative integers such that for \( i=0,\dots,n+1 \),
  \begin{itemize}
  \item if \( a_i\equiv 1\pmod\ell \), then \( a_i \) is a valley, that is, \( a_{i-1}>a_i<a_{i+1} \),
  \item if \( a_i\equiv -1\pmod\ell \), then \( a_i \) is a peak, that is, \( a_{i-1}<a_i>a_{i+1} \),
  \end{itemize}
  where we set \( a_0=r \) and \( a_{n+1} = s \). Here, we ignore the
  inequalities involving \( a_t \) for \( t=-1 \) or \( t=n+2 \). Denote by \(
  \widetilde{\PV}_{n,r,s}^{\ell,k} \) the set of modified \( (\ell,r,s)
  \)-peak-valley sequences \( ( a_1,\dots,a_{n}) \) with bound \( k \), i.e., \(
  0\le a_i\le k \) for all \( i=0,\dots,n+1 \).
\end{defn}

\begin{thm}\label{thm:rsneg3k}
  Suppose that \( \vb=(b_i)_{i\ge0} \) and \( \vla=(\lambda_i)_{i\ge1} \) are
  the sequences given by \( b_i=-V_i^{-1} \) and \(
  \lambda_{i}=V_{i}^{-1}V_{i-1}^{-1} \) for all \( i \). Then 
  \[ 
    \mu_{-n,r,s}^{\le 3k}(\vb,\vla) =(-1)^{\flr{{(r+1)}/{3}}+\flr{{(s+1)}/{3}}+n}\frac{V_0\cdots V_s}{V_0\cdots V_{r-1}}\sum_{\pi\in\widetilde{\PV}_{n-1,r,s}^{3,3k}}\wt(\pi).
  \]
\end{thm}
\begin{proof}
 This can be proved by the same arguments as in the proof of Theorem~\ref{thm:rsneg3k-1}.
 We omit the details. 
\end{proof}

Putting \( V_i=-1 \) in Theorem~\ref{thm:rsneg3k} we obtain the following
corollary.

\begin{cor} 
We have
  \[ 
    \mu_{-n,r,s}^{\le 3k}(\vo,\vo)=(-1)^{\flr{{(r+1)}/{3}}+\flr{{(s+1)}/{3}}+r+s}\left|\widetilde{\PV}_{n-1,r,s}^{3,3k}\right|.
  \]
\end{cor}

\section{A general reciprocity theorem}
\label{sec:gener-recipr-theor}

In this section we prove a general reciprocity theorem, Theorem~\ref{thm:
  CKconjecturemain}. Using this theorem we will prove the Cigler--Krattenthaler
conjectures, Theorems~\ref{thm:CKconjecture1} and \ref{thm:CKconjecture2} in the
next section.

Following the notation in \cite{Kratt_Hankel1}, let $R^{(n)}$ be the operator
defined on polynomials in $b_i$'s and $\lambda_i$'s that replaces each $b_i$ by
$b_{n-i}$ and each $\lambda_i$ by $\lambda_{n+1-i}$. For example,
$R^{(5)}(b_1+\lambda_2+b^2_3\lambda_1)=b_4+\lambda_4+b^2_2\lambda_5$.

Recall the matrix $A^{\leq k} (\vb,\vla)$ given in \eqref{eq:Ak(b,la)}.
We now state the general reciprocity theorem.

\begin{thm}\label{thm: CKconjecturemain}
   For positive
  integers $k$ and $m$, we have
  \begin{multline*}
    \det\left( \mu^{\leq k+m-1}_{n+i+j+2m-2}(\vb,\vla) 
    \right)_{i,j=0}^{k-1}\\
    = \left(\prod_{i=1}^{k+m-1}\lambda_i^{k-i}\right)\det\left( A^{\leq k+m-1} (\vb,\vla) \right)^{n+2m-2}
      R^{(k+m-1)}\left(
   \det
    \left(\mu^{\leq k+m-1}_{-n-i-j}(\vb,\vla)\right)_{i,j=0}^{m-1}\right).
  \end{multline*}
\end{thm}  

Before proving Theorem \ref{thm: CKconjecturemain}, we define several terminologies and prove some auxiliary results.

\begin{defn}
  For a matrix \( A=(A_{i,j})_{i,j=0}^k\)
  and two subsets \( I,J\subseteq \{0,\dots,k\} \) of the same cardinality,
  we define
  \[
    [A]_{I,J} = \det(A_{i,j})_{i\in I, j\in J}.
  \]
\end{defn}

The following well-known lemma is an important tool in our proofs.

\begin{lem}\label{lem:inverse_matrix}
  Suppose that \( A=(A_{i,j})_{i,j=0}^{k}\) is an invertible matrix. For subsets
  \( I,J\subseteq \{0,\dots,k\} \) of the same cardinality, we have
  \[
    \left[A^{-1}\right]_{I, J}=(-1)^{\norm I+\norm J} \frac{[A]_{J^{\prime}, I^{\prime}}}{\det (A)},
  \]
  where \( \norm I=\sum_{i\in I}i \) and \( I'=\{0,\dots,k\}\setminus I \).
\end{lem}

\begin{defn}
  Let \( A=(A_{i,j})_{i,j=0}^{k} \). We define the weighted directed graph
  $G(A)$ whose vertex set is $V(A)=\{(i,j): i\in\ZZ, j\in \{0,\dots,k\}\}$ and
  edge set is $E(A)=\{(i,j)\rightarrow (i+1,j'): i\in\ZZ, j,j'\in \{0,\dots,k\}\}$.
  We assign the weight $A_{j,j'}$ to each edge $(i,j)\rightarrow (i+1,j')$ and
  ignore the edges with zero weights.

  For \( u,v\in V(A) \), let $P(G(A);u\to v)$ be the set of paths in $G(A)$ from
  $u$ to $v$. The weight $\wt_A(\pi)$ of a path $\pi$ is defined to be the
  product of weights on its edges. For \( u_0,\dots,u_n,v_0,\dots,v_n\in V(A)
  \), we define $\NI\left(G(A);(u_0,\dots,u_n) \rightarrow
    (v_0,\dots,v_n)\right)$ to be the set of tuples $\vpi=(\pi_0,\dots,\pi_n)$
  of nonintersecting paths, i.e., no two paths meet at a vertex in \( V(A) \),
  such that each $\pi_i$ is a path in \( G(A) \) from $u_i$ to $v_j$ for some
  $j$. For such a path tuple $\vpi$, there exists a permutation $\sigma$ of \(
  \{0,\dots,n\} \) such that each $\pi_i$ is a path from \( u_i \) to
  $v_{\sigma(i)}$. The weight $\wt_A(\vpi)$ of the path tuple is
  defined to be $\sgn(\sigma)\prod_{i=0}^{n} \wt_A(\pi_i)$.
\end{defn}

Note that if \( A \) is the tridiagoanl matrix $A^{\leq k}(\vb,\vla)$, then, for
any \( i\in \ZZ \), a path \( \pi \) in \( P(G(A);(i,0)\to (i+n,0)) \) can
be identified with a Motzkin path in \( \Mot^{\le k}_n \) and we have
\begin{equation}
  \label{eq:mun=wt_A(pi)}
  \mu_{n}^{\le k}(\vb,\vla)=
  \sum_{\pi\in \Mot_{n}^{\le k}} \wt(\pi)=\sum_{\pi \in P(G(A);(i,0)\rightarrow (i+n,0))}\wt_A(\pi).
\end{equation}
Moreover, by Proposition~\ref{prop: matrix interpretation of negative moment},
\begin{equation}
  \label{eq:mu-n=wt_A(pi)}
  \mu_{-n}^{\le k}(\vb,\vla)=
  \epsilon_0^T \left( A^{-1} \right)^{n} \epsilon_0
  =\sum_{\pi \in P(G(A^{-1});(i,0)\rightarrow (i+n,0))}\wt_{A^{-1}}(\pi).
\end{equation}

\begin{lem}\label{lemma: tridiagoanl minor property}
  For the tridiagoanl matrix $A=A^{\leq k}(\vb,\vla)$, the following
  statements hold.
  \begin{enumerate}
    \item Given $I ,J \subseteq \{0,\dots,k\}$ of the same cardinality and 
    $\{0,\dots,i\}\subseteq I$ for some $i$, if $J$ misses two or more elements in 
    $\{0,\dots,i+1\}$, then $[A]_{I,J}=0$.
    \item Given $I\subseteq \{0,\dots,i-1\}$, $J\subseteq \{0,\dots,i\}$ of the same cardinality, we have 
      \begin{align*}
        [A]_{I\cup \{i,\dots,k-1\},J\cup \{i+1,\dots,k\}}&=[A]_{I,J} ,\\
        [A]_{I\cup \{i+1,\dots,k\},J\cup \{i,\dots,k-1\}}&=[A]_{I,J} \prod_{j=i+1}^{k}\lambda_j.
      \end{align*}
  \end{enumerate}
\end{lem}
\begin{proof}

(1) Let $B$ be the submatrix of \( A \) consisting of the rows indexed by $0,\dots,i$.
Since $A$ is a tridiagonal matrix, the $j$th column of $B$ is zero if $j>i+1$. So if $J$
 misses two or more elements in $0,\dots,i+1$, the submatrix of $B$ consisting of
 the columns indexed by $J$ has rank at most $i$. We conclude $[A]_{I,J}=0$. 

 (2) The submatrix of $A$ with rows indexed by $i,\dots,k-1$ and columns indexed by $i+1,\dots,k$
 is a lower triangular matrix with diagonal entries all 1. So the first identity follows.
 Likewise, the submatrix of $A$ with rows indexed by $i+1,\dots,k$ and columns indexed by $i,\dots,k-1$
 is an upper triangular matrix with diagonal entries $\lambda_{i+1},\dots,\lambda_{k}$. This
 gives the second identity.  
\end{proof}

\begin{lem}\label{lemma: thm7.3 lhs}
  Letting $A=A^{\leq k+m-1} (\vb,\vla)$, we have
  \begin{align*}
    \det\left( \mu^{\leq k+m-1}_{n+i+j+2m-2}(\vb,\vla) \right)_{i,j=0}^{k-1}=
    \left(\prod_{i=1}^{k-1}\lambda_i^{k-i}\right) \sum_{(I_0,\dots,I_{n+2m-2})\in X} \prod_{j=0}^{n+2m-3}[A]_{I_j,I_{j+1}},
  \end{align*}
  where $X$ is the set of all tuples $(I_0,\dots,I_{n+2m-2})$ of $k$-element
  subsets of $\{0,\dots,k+m -1\}$ such that $I_j, I_{n+2m-2-j} \subseteq
  \{0,\dots,k-1+j\}$ for all $0\le j\leq m-1$.
\end{lem}
\begin{proof}  
  By \eqref{eq:mun=wt_A(pi)}, we have
  \begin{equation*}
    \mu_{n+i+j+2m-2}^{\leq k+m-1}(\vb,\vla)=
    \sum_{\pi \in P(G(A);(-i,0)\rightarrow (n+2m-2+j,0))}\wt_A(\pi).
  \end{equation*}
  Thus, the Lindstr\"om--Gessel--Viennot lemma \cite{LGV,Lindstrom} gives
  \begin{equation*}
    \det\left( \mu^{\leq k+m-1}_{n+i+j+2m-2}(\vb,\vla) \right)_{i,j=0}^{k-1}
    =\sum_{\vpi \in \NI(G(A);R \rightarrow S)}\wt_A(\vpi),
  \end{equation*}
  where $R=((0,0),(-1,0),\dots,(-k+1,0))$ and
  $S=((n+2m-2,0),(n+2m-1,0),\dots,(n+2m+k-3,0))$.

  Suppose \( \vpi=(\pi_0,\dots,\pi_{k-1}) \in \NI(G(A);R \rightarrow S) \).
  Since \( \pi_0,\dots,\pi_{k-1} \) are nonintersecting Motzkin paths and each
  \( \pi_i \) is from \( (-i,0) \) to \( (n+2m-2+\sigma(i),0) \), for some
  permutation \( \sigma \), the first \( i \) steps (resp.~last \( \sigma(i) \)
  steps) of \( \pi_i \) are up steps (resp.~down steps) whose weights are \( 1
  \)'s (resp.~\( \lambda_1,\dots,\lambda_{\sigma(i)} \)). Considering the
  subpath obtained from \( \pi_i \) by deleting the first \( i \) steps and last
  \( \sigma(i) \) steps, we obtain
  \begin{equation}\label{eq:conjectureproof1}
    \det\left( \mu^{\leq k+m-1}_{n+i+j+2m-2}(\vb,\vla) \right)_{i,j=0}^{k-1}
    =\left(\prod_{i=1}^{k-1}\lambda_i^{k-i}\right)\sum_{
      \vpi\in \NI(G(A);R_1 \rightarrow S_1)}\wt_A(\vpi),
  \end{equation} 
  where $R_1=((0,0),(0,1),\dots,(0,k-1))$ and
  $S_1=((n+2m-2,0),(n+2m-2,1),\dots,(n+2m-2,k-1))$.
  
  For each $\vpi=(\pi_0,\dots,\pi_{k-1}) \in \NI(G(A);R_1 \rightarrow S_1)$, we
  define $I(\vpi)=(I_0,\dots,I_{n+2m-2})$, where each $I_j$ is the $k$-element
  subset of $\{0,\dots,k+m -1\}$ consisting of the $y$-coordinates of the points
  of $\pi_0,\dots,\pi_{k-1}$ on the line $x=j$. For brevity we write \(
  \vI=(I_0,\dots,I_{n+2m-2}) \). Observe that, since $\vpi \in \NI(G(A);R_1
  \rightarrow S_1)$, we have $I_0=I_{n+2m-2}=\{0,\dots,k-1\}$. Moreover, since
  $\pi_0,\dots,\pi_{k-1}$ are Motzkin paths, we also have $I_j, I_{n+2m-2-j}
  \subseteq \{0,\dots,k-1+j\}$ for all $0\le j\leq m-1$, that is, \( \vI\in X
  \). Therefore we can rewrite \eqref{eq:conjectureproof1} as
  \begin{equation}
    \label{eq:X}
    \det\left( \mu^{\leq k+m-1}_{n+i+j+2m-2}(\vb,\vla)
    \right)_{i,j=0}^{k-1}
    = \left(\prod_{i=1}^{k-1}\lambda_i^{k-i}\right)
    \sum_{\vI\in X}
    \sum_{\substack{\vpi \in \NI(G(A);R_1 \rightarrow S_1)\\
        I(\vpi)=\vI}}
    \wt_{A}(\vpi).
  \end{equation}

  For a fixed tuple \( \vI=(I_0,\dots,I_{n+2m-2})\in X \), applying the
 Lindstr\"om--Gessel--Viennot lemma repeatedly to the paths (of length \( 1 \)) starting
  from the points \( (j,r) \) for \( r\in I_j \) to the points \( (j+1,s) \) for
  \( s\in I_{j+1} \), we obtain
  \begin{equation}\label{eq:conjectureproof2}
    \sum_{\substack{\vpi \in \NI(G(A);R_1 \rightarrow S_1)\\
        I(\vpi)=\vI}}    \wt_{A}(\vpi)= \prod_{j=0}^{n+2m-3}[A]_{I_j,I_{j+1}}.
  \end{equation}
  Combining \eqref{eq:X} and \eqref{eq:conjectureproof2} completes the proof.
\end{proof}

\begin{lem}\label{lemma: thm7.3 rhs}
  Letting $A=A^{\leq k+m-1} (\vb,\vla)$, we have
  \begin{multline*}
   \det \left(\mu^{\leq k+m-1}_{-n-i-j}(\vb,\vla)\right)_{i,j=0}^{m-1}\\
    = \sum_{(J_0,\dots,J_{n+2m-2})\in Y}   \prod_{j=0}^{m-2}[A^{-1}]_{J_j,J_{j+1}\setminus\{0\}}
    \prod_{j=m-1}^{n+m-2}[A^{-1}]_{J_j,J_{j+1}}
    \prod_{j=n+m-1}^{n+2m-3}[A^{-1}]_{J_j\setminus\{0\},J_{j+1}},
  \end{multline*}
  where $Y$ is the set of all tuples $(J_0,\dots,J_{n+2m-2})$ of subsets of
  $\{0,\dots,k+m -1\}$ satisfying the following conditions:
  \begin{enumerate}
  \item Both $J_j$ and $J_{n+2m-2-j}$ have cardinality \( j+1 \) and contain \( 0 \)
    for all $0\leq j\leq m-1$. 
  \item \( |J_j|=m \) for all $m-1\leq j\leq n+m-1$.
  \item $J_j\cap\{1,\dots,m-1-j\}= J_{n+2m-2-j}\cap\{1,\dots,m-1-j\}= \emptyset $ for all $0\le j\leq
    m-2$.
  \end{enumerate}

\end{lem}
\begin{proof}
  By \eqref{eq:mu-n=wt_A(pi)} and the Lindstr\"om--Gessel--Viennot lemma, we have
  \begin{equation}\label{eq:conjectureproofe}
    \det\left( \mu^{\leq k+m-1}_{-n-i-j}(\vb,\vla) \right)_{i,j=0}^{m-1}
    =\sum_{\vpi \in \NI(G(A^{-1});R \rightarrow S)}\wt_{A^{-1}}(\vpi),
  \end{equation}
  where $R=((m-1,0),(m-2,0),\dots,(0,0))$ and
  $S=((n+m-1,0),(n+m,0),\dots,(n+2m-2,0))$.

  For each \( \vpi=(\pi_0,\dots,\pi_{m-1}) \in \NI(G(A^{-1});R \rightarrow S) \),
  we define $J(\vpi)=(J_0,\dots,J_{n+2m-2})$, where $J_j$ is the subset of
  $\{0,\dots,k+m -1\}$ consisting of the $y$-coordinates of the points of
  $\pi_0,\dots,\pi_{m-1}$ on the line $x=j$. For brevity we write \(
  \vJ=(J_0,\dots,J_{n+2m-2}) \). It is easy to check that the tuple \( \vJ \)
  satisfies the first two conditions for the elements in \( Y \). Let \( Z \) be
  the set of such tuples \( \vJ \). Then we can rewrite
  \eqref{eq:conjectureproofe} as
  \begin{equation}\label{eq:conjectureproofe2}
    \det\left( \mu^{\leq k+m-1}_{-n-i-j}(\vb,\vla) \right)_{i,j=0}^{m-1}
    =\sum_{\vJ\in Z} \sum_{\substack{\vpi \in \NI(G(A^{-1});R \rightarrow S)\\
        J(\vpi)=\vJ}}\wt_{A^{-1}}(\vpi).
  \end{equation}
 
  For a fixed tuple \( \vJ=(J_0,\dots,J_{n+2m-2})\in Z \), applying the
 Lindstr\"om--Gessel--Viennot lemma repeatedly to the paths (of length \( 1 \)) starting
  from the points \( (j,r) \) for \( r\in J_j \) to the points \( (j+1,s) \) for
  \( s\in J_{j+1} \), where we ignore the point \( (j+1,0) \) (resp.~\( (j,0)
  \)) if \( 0\le j\le m-2 \) (resp.~\( n+m-1\le j\le n+2m-3 \)), we obtain
  \begin{equation}\label{eq:NIJ_wt}
    \sum_{\substack{\vpi \in \NI(G(A^{-1});R \rightarrow S)\\
        J(\vpi)=\vJ}}\wt_{A^{-1}}(\vpi) 
  = w(\vJ),
  \end{equation}
  where
  \[
    w(\vJ) = \prod_{j=0}^{m-2}[A^{-1}]_{J_j,J_{j+1}\setminus\{0\}}
    \prod_{j=m-1}^{n+m-2}[A^{-1}]_{J_j,J_{j+1}}
    \prod_{j=n+m-1}^{n+2m-3}[A^{-1}]_{J_j\setminus\{0\},J_{j+1}}.
  \]

  Combining \eqref{eq:conjectureproofe2} and \eqref{eq:NIJ_wt} gives
  \[
    \det\left( \mu^{\leq k+m-1}_{-n-i-j}(\vb,\vla) \right)_{i,j=0}^{m-1} 
    =\sum_{\vJ\in Z} w(\vJ).
  \]
  Since \( Z \) is the set of tuples \( (J_0,\dots,J_{n+2m-2}) \) satisfying the
  conditions (1) and (2) (but not necessarily (3)) for the elements in \( Y \),
  it remains to show the following claim.

  \textbf{Claim:} Let \( \vJ=(J_0,\dots,J_{n+2m-2})\in Z \). If
  $J_j\cap\{1,\dots,m-1-j\}\ne \emptyset $ or $J_{n+2m-2-j}\cap\{1,\dots,m-1-j\}\ne
  \emptyset $ for some $0\le j\leq m-2$, then \( w(\vJ)=0 \).

  To prove the claim suppose that $J_j\cap\{1,\dots,m-1-j\}\neq \emptyset$ for
  some $0\le j\leq m-2$. Take the largest \( j \) so that
  $J_{j+1}\cap\{1,\dots,m-2-j\}= \emptyset$, which is clearly true if \( j=m-2
  \). By Lemma~\ref{lem:inverse_matrix}, we have
  \begin{equation*}
    [A^{-1}]_{J_j,J_{j+1}\setminus\{0\}}= (-1)^{\norm{J_j}+\norm{J_{j+1}\setminus\{0\}}}
    \frac{[A]_{(J_{j+1}\setminus\{0\})',(J_j)'}}{\det(A)}.
  \end{equation*}
  By the assumption on \( j \) we have $\{0,\dots,m-2-j\}\subseteq
  (J_{j+1}\setminus\{0\})'$ and $0,t\not\in (J_j)'$ for some $t\in
  \{1,\dots,m-1-j\}$. Thus, by Lemma~\ref{lemma: tridiagoanl minor property} (1),
  we have $[A]_{(J_{j+1}\setminus\{0\})',(J_j)'}=0$, which implies \(
  w(\vJ)=0 \). Similarly, one can prove that if
  $J_{n+2m-2-j}\cap\{1,\dots,m-j-1\}\ne \emptyset $ for some $0\le j\leq m-2$, then
  \( w(\vJ)=0 \). This settles the claim and the proof is
  completed.
\end{proof}

We are now ready to prove Theorem \ref{thm: CKconjecturemain}.

\begin{proof}[Proof of Theorem \ref{thm: CKconjecturemain}]
  Abusing the notation, let $R^{(k+m-1)}$ also denote the operator acting on the
  subsets $K$ of $\{0,\dots,k+m-1\}$ by
\begin{equation*}
  R^{(k+m-1)}(K)=\{k+m-1-i: i \in K\}.
\end{equation*}

Recall the sets \( X \) and \( Y \) given in Lemmas~\ref{lemma: thm7.3 lhs} and
\ref{lemma: thm7.3 rhs} respectively. We define the map $f:X \rightarrow Y$
by \( f(I_0,\dots,I_{n+2m-2})=(J_0,\dots,J_{n+2m-2}) \), where
\begin{equation}
  \label{eq:J}
  J_j = 
  \begin{cases} 
    \{0\}\cup R^{(k+m-1)}(\{0,\dots,k-1+j\}\setminus I_j) 
    &\mbox{if \(0 \leq j \leq m-1 \),}\\
    R^{(k+m-1)}(\{0,\dots,k+m-1\}\setminus I_j)  
    &\mbox{if \(m-1 \leq j \leq n+m-1 \),}\\
    \{0\}\cup R^{(k+m-1)}(\{0,\dots,k+n+2m-3-j\}\setminus I_j)  
    &\mbox{if \(n+m-1 \leq j \leq n+2m-2 \).}
  \end{cases}
\end{equation}
Note that \( J_j \) is well defined when \( j=m-1 \) or \( j=n-m-1 \). It is not
hard to see that the map $f$ is a bijection.

We claim that, for $0 \leq j \leq m-2$,
\begin{equation}\label{eq:claim_A-1}
  [A^{-1}]_{J_j,J_{j+1}\setminus\{0\}}
  =  (-1)^{\norm{J_j}+\norm{J_{j+1}\setminus\{0\}}}
  \frac{R^{(k+m-1)}([A]_{I_{j+1},I_j})}{\det(A)}.
\end{equation}
To prove the claim, we first use Lemma \ref{lem:inverse_matrix} to obtain
\begin{equation}\label{eq:J=J'}
  [A^{-1}]_{J_j,J_{j+1}\setminus\{0\}}
  =(-1)^{\norm{J_j}+\norm{J_{j+1}\setminus\{0\}}}
  \frac{[A]_{(J_{j+1}\setminus\{0\})',(J_j)'}}{\det(A)}. 
\end{equation}
Note that, for $0 \leq j \leq m-2$, \eqref{eq:J} implies
\begin{align*}
  (J_j)'&=\{0,\dots,k+m-1\}\setminus J_j=R^{(k+m-1)}(I_j\cup \{k+j,\dots,k+m-2\}),\\
  (J_{j+1}\setminus\{0\})'
        &=\{0,\dots,k+m-1\}\setminus (J_{j+1}\setminus\{0\})\\
        &= R^{(k+m-1)}(I_{j+1}\cup \{k+j+1,\dots,k+m-1\}).
\end{align*}
Thus, the right-hand side of \eqref{eq:J=J'} is equal to
\begin{multline*}
  (-1)^{\norm{J_j}+\norm{J_{j+1}\setminus\{0\}}}
    \frac{[A]_{R^{(k+m-1)}(I_j\cup \{k+j,\dots,k+m-2\}),R^{(k+m-1)}(I_{j+1}\cup \{k+j+1,\dots,k+m-1\})}}{\det(A)}\\
  = (-1)^{\norm{J_j}+\norm{J_{j+1}\setminus\{0\}}}
    \frac{R^{(k+m-1)}([A]_{I_j\cup \{k+j,\dots,k+m-2\}),I_{j+1}\cup \{k+j+1,\dots,k+m-1\}})}{\det(A)},
\end{multline*}
which, by Lemma \ref{lemma: tridiagoanl minor property} (2), is equal to the
right-hand side of \eqref{eq:claim_A-1} and the claim is proved.

A similar argument shows that,
for $m-1\leq j\leq n+m-2$,
\begin{equation}\label{eq:claim_A-2}
  [A^{-1}]_{J_j,J_{j+1}}
  =  (-1)^{\norm{J_j}+\norm{J_{j+1}}}
  \frac{R^{(k+m-1)}([A]_{I_{j+1},I_j})}{\det(A)},
\end{equation}
and, for $n+m-1\leq j\leq n+2m-3$,
\begin{equation}\label{eq:claim_A-3}
  [A^{-1}]_{J_j\setminus\{0\},J_{j+1}}=(-1)^{\norm{J_j\setminus\{0\}}+\norm{J_{j+1}}}
  \frac{\left(\prod_{i=1}^{j-n-m+1}\lambda_{k+m-i}\right)
  R^{(k+m-1)} ([A]_{I_{j+1},I_j})}
  {\det(A)}.
\end{equation}

By \eqref{eq:claim_A-1}, \eqref{eq:claim_A-2}, and \eqref{eq:claim_A-3}, and
using the fact that \( \norm{J_0}=\norm{J_{n+2m-2}}=0 \) and \(
\norm{J\setminus\{0\}} = \norm{J} \) for any set \( J \), we obtain
\begin{multline}\label{eq:conjectureproof4}
  \prod_{j=0}^{m-2}[A^{-1}]_{J_j,J_{j+1}\setminus\{0\}}
  \prod_{j=m-1}^{n+m-2}[A^{-1}]_{J_j,J_{j+1}}
  \prod_{j=n+m-1}^{n+2m-3}[A^{-1}]_{J_j\setminus\{0\},J_{j+1}}\\
  =\frac{\prod_{i=k}^{k+m-1}\lambda_i^{k-i}}{\det(A)^{n+2m-2}}
  R^{(k+m-1)} \left( \prod_{j=1}^{n+2m-2}[A]_{I_{j+1},I_{j}} \right).
\end{multline}
Since \( (I_0,\dots,I_{n+2m-2})\in X \) if and only if \(
(I_{n+2m-2},\dots,I_{0})\in X \), combining Lemma \ref{lemma: thm7.3 lhs}, Lemma
\ref{lemma: thm7.3 rhs} and \eqref{eq:conjectureproof4} completes the proof.
\end{proof}

\section{Proof of Cigler--Krattenthaler conjectures}
\label{sec:conj}

In this section we show that Theorem~\ref{thm: CKconjecturemain} implies the
Cigler--Krattenthaler conjectures, Theorems~\ref{thm:CKconjecture1} and
\ref{thm:CKconjecture2}. To obtain Theorem \ref{thm:CKconjecture2}, we can
simply put $\vb=\vo$ and $\vla=\vo$ in Theorem~\ref{thm: CKconjecturemain}.
However, it is more difficult to derive Theorem~\ref{thm:CKconjecture1} from
Theorem~\ref{thm: CKconjecturemain} because the left-hand side of the equation
in Theorem~\ref{thm:CKconjecture1} is not of the form as written in
Theorem~\ref{thm: CKconjecturemain}. To remedy this we find suitable sequences
$\vb$ and $\vla$ such that
\begin{align}\label{eq: b lambda special condition}
  &\mu^{\leq k+m-1}_{n+i+j+2m-2}(\vb,\vla)=
  \sum_{s=0}^{2k+2m-1}\mu_{n+i+j+2m-1,0,s}^{\leq 2k+2m-1}(\vz,\vo).
\end{align}
Indeed we will show in Lemma~\ref{lemma: conj 50 connection 1} that \eqref{eq: b
  lambda special condition} holds if \( \vb=\vb^{(k+m-1)} \) and \( \vla=-\vo \), where
the sequences \( \vb^{(\ell)} \) and \( -\vo \) are defined by 
\begin{align}
  \label{eq:tvb}
  \vb^{(\ell)}&=\left( b^{(\ell)}_i \right)_{i\ge0}, \qquad b^{(\ell)}_i =
        \begin{cases}
          (-1)^{i}2 & \mbox{if \( 0 \leq i<\ell \),}\\
          (-1)^{i}  & \mbox{if \( i\ge \ell \),}
        \end{cases}\\
  \label{eq:-vo}
  -\vo&=(-1,-1,\dots).
\end{align}

We will mostly consider $\mu^{\le \ell}_n(\vb^{(\ell)},\vla)$, which does not
depend on the values $b^{(\ell)}_i$ for $i>\ell$.

\begin{remark}
  One may wonder how to guess that $\vb=\vb^{(k+m-1)}$ and $\vla=-\vo$ is a solution to
  \eqref{eq: b lambda special condition}. Such a solution can be found by
  computer once we fix the value of \( k+m \). For example, if \( k+m=3 \), then
  the sequences \( \vb \) and \( \vla \) must satisfy
  \begin{equation}\label{eq:mu2}
    \mu^{\leq 2}_{n}(\vb,\vla)= \sum_{s=0}^{5}\mu_{n+1,0,s}^{\leq 5}(\vz,\vo).
  \end{equation}
  Substituting \( n=1,\dots,5 \) in \eqref{eq:mu2} gives five equations with
  variables \( b_0,b_1,b_2, \lambda_1,\lambda_2\). One can check by computer
  that there is a unique solution to these equations, which is \( b_0=2, b_1=-2,
  b_2=1, \lambda_1=-1, \lambda_2=-1 \). After computing more solutions for
  different choices of \( k+m \) one can guess that $\vb=\vb^{(k+m-1)}$ and $\vla=-\vo$
  is a solution to \eqref{eq: b lambda special condition}.
\end{remark}

Now we begin with a simple lemma.

\begin{lem}\label{lem: special matrix determinant}We have
  \begin{align}\label{eq: special det first}
    \det\left(A^{\leq k-1}(\vo,\vo)\right)
    &=\begin{cases} 0 &
      \mbox{if \( k\equiv 2 \pmod 3 \)},\\
      (-1)^{\flr{{k}/{3}}} &\mbox{otherwise},
    \end{cases}\\
    \label{eq: special det second}
    \det\left(A^{\leq k-1}(\vb^{(k-1)},-\vo)\right)
    &=(-1)^{\flr{{k}/{2}}}.
  \end{align}
\end{lem}
\begin{proof} To prove \eqref{eq: special det first}, we expand the determinant with respect to the first row to get 
  \begin{align*}
    \det\left(A^{\leq k-1}(\vo,\vo)\right)=\det\left(A^{\leq k-2}(\vo,\vo)\right)-
    \det\left(A^{\leq k-3}(\vo,\vo)\right).
  \end{align*}
  Then \eqref{eq: special det first} follows easily by induction on $k$.

  For the second identity, let $U=(U_{i,j})_{i,j=0}^{k-1}$ and
  $L=(L_{i,j})_{i,j=0}^{k-1}$ be the matrices defined by
\begin{align*}
  U_{i,j} =\begin{cases}
    1 &\text{ if }\qquad i=j,\\
    (-1)^{i-1} &\text{ if }\qquad i=j-1,\\
    0 &\text{otherwise},
   \end{cases} \qquad L_{i,j} =\begin{cases}
    (-1)^{i} &\text{ if }\qquad i=j,\\
    -1 &\text{ if }\qquad i=j+1,\\
    0 &\text{otherwise}.
   \end{cases}
  \end{align*}
  It is easy to check that $A^{\leq k-1}(\vb^{(k-1)},-\vo)=UL$. Since $U$ is an
  upper-triangular matrix with diagonal entries all \( 1 \) and $L$ is a
  lower-triangular matrix with diagonal entries $1,-1,1,-1,\dots$, we obtain
  \eqref{eq: special det second}.
\end{proof}

Since the proof of Theorem \ref{thm:CKconjecture2} is simpler than that of
Theorem~\ref{thm:CKconjecture1} we present it first.

  \begin{proof}[Proof of Theorem \ref{thm:CKconjecture2}]
    We put $\vb=\vo$ and $\vla=\vo$ in Theorem \ref{thm: CKconjecturemain}.
    Then Theorem \ref{thm:CKconjecture2} immediately follows from \eqref{eq:
      special det first} and
   \begin{equation*}
    \left. R^{(k+m-1)}\left(\mu_{-n-i-j}^{\le k+m-1}(\vb,\vla)\right) \right|_{\vb=\vo,\vla=\vo} =\mu_{-n-i-j}^{\le k+m-1}(\vo,\vo).
    \qedhere
   \end{equation*} 
  \end{proof}

  In order to prove Theorem~\ref{thm:CKconjecture1} we need the following two
  technical lemmas whose proofs will be given later.

  \begin{lem}\label{lemma: conj 50 connection 1}
    For $n\geq 0$ and $k\ge1$, we have
    \begin{align*}
     \mu^{\leq k}_{n}(\vb^{(k)},-\vo) =\sum_{s=0}^{2k+1}\mu_{n+1,0,s}^{\leq 2k+1}(\vz,\vo).
    \end{align*}
  \end{lem}

  \begin{lem}\label{lemma: conj 50 connection 2}
    For $n\geq 0$ and $k\ge1$, we have
    \begin{align*}
      \left.  (-1)^{kn}\left(R^{(k)}
      \left( \mu^{\leq k}_{-n}(\vb,\vla) \right)\right)\right|_{\vb=\vb^{(k)}, \vla=-\vo}
      =\left| \Alt_{n}^{\leq k+1} \right|.
    \end{align*}
  \end{lem}

  \begin{proof}[Proof of Theorem~\ref{thm:CKconjecture1}] 
    We put $\vb=\vb^{(k+m-1)}$ and $\vla=-\vo$ in Theorem \ref{thm:
      CKconjecturemain} and apply Lemma~\ref{lemma: conj 50 connection 1},
    \eqref{eq: special det second}, and Lemma~\ref{lemma: conj 50 connection 2},
    which gives
    \begin{multline}\label{eq:proof_step1}
      \det\left( \sum_{s=0}^{2k+2m-1}\mu_{n+i+j+2m-1,0,s}^{\leq 2k+2m-1}(\vz,\vo) 
        \right)_{i,j=0}^{k-1}\\
        = \left(\prod_{i=1}^{k+m-1}(-1)^{k-i}\right)
        (-1)^{\flr{\frac{k+m}{2}}(n+2m-2)}
       \det
        \left((-1)^{(k+m-1)(n+i+j)}|\Alt_{n+i+j}^{\leq k+m}|\right)_{i,j=0}^{m-1}.
    \end{multline}
    Observe that $\flr{\frac{k+m}{2}} \equiv \binom{k+m}{2} \pmod 2$ and
    \begin{equation*}
      \det
        \left((-1)^{(k+m-1)(n+i+j)}|\Alt_{n+i+j}^{\leq k+m}|\right)_{i,j=0}^{m-1}=(-1)^{(k+m-1)nm}\det
        \left(|\Alt_{n+i+j}^{\leq k+m}|\right)_{i,j=0}^{m-1},
    \end{equation*}
    and \( (-1)^{(k+m-1)nm} = (-1)^{kmn}\). Therefore we
    can rewrite \eqref{eq:proof_step1} as
    \[
      \det\left( \sum_{s=0}^{2k+2m-1}\mu_{n+i+j+2m-1,0,s}^{\leq 2k+2m-1}(\vz,\vo) 
      \right)_{i,j=0}^{k-1}\\
      = s \cdot   
      \det
      \left(|\Alt_{n+i+j}^{\leq k+m}|\right)_{i,j=0}^{m-1},
    \]
    where \( s = (-1)^{k(k+m-1)-\binom{k+m}{2}}(-1)^{n\binom{k+m}{2}}(-1)^{knm}
    \), which is equal to
    \[
      (-1)^{km+(n+1)\binom{k+m}{2}+knm}  
      = (-1)^{\left(km +\binom{k+m}{2}  \right)(n+1)}
      = (-1)^{\left(\binom{k}{2} + \binom{m}{2}  \right)(n+1)}. 
    \]
     This completes the proof.
   \end{proof}
   
   Now it remains to prove Lemmas~\ref{lemma: conj 50 connection 1} and
   \ref{lemma: conj 50 connection 2}.

\begin{proof}[Proof of Lemma~\ref{lemma: conj 50 connection 1}]
  Using \eqref{eq:matrix_mu} we can restate the lemma as
  \begin{equation}\label{eq:e0e0=e0v}
\epsilon_0^{T} (A^{\leq k}(\vb^{(k)},-\vo))^{n} \epsilon_0=\epsilon_0^{T} (A^{\leq 2k+1}(\vz,\vo))^{n+1} v ,
  \end{equation}
  where $v=\sum_{s=0}^{2k+1}\epsilon_s$.

  Fix the integer \( k\ge1 \) and, for \( n\ge0 \), let
  \begin{align*}
    (A^{\leq k}(\vb^{(k)},-\vo))^{n} \epsilon_0&=(a_{n,0},\dots,a_{n,k})^{T},\\
    (A^{\leq 2k+1}(\vz,\vo))^{n} v&=(d_{n,0},\dots,d_{n,2k+1})^{T}.
  \end{align*}
  We also define $a_{n,j}=0$ for \( j\not\in\{0,\dots,k\} \) and $d_{n,i}=0$ for
  \( i\not\in \{0,\dots,2k+1\} \). Then, by definition, we have
  \begin{align}
    \label{eq:a1}
    a_{n+1,i}  &= -a_{n,i-1} +(-1)^i2 a_{n,i}+a_{n,i+1}, \qquad 0\le i\le k-1,\\
    \label{eq:a2}
   a_{n+1,k}  &= -a_{n,k-1} +(-1)^k a_{n,k},\\
    \label{eq:d1}
    d_{n+1,i}  &= d_{n,i-1} + d_{n,i+1}, \qquad 0\le i\le 2k+1,\\
    \label{eq:d2}
    d_{n+1,i}&=d_{n+1,2k+1-i}, \qquad 0\le i\le 2k+1,
  \end{align}
  where \eqref{eq:d2} follows from the symmetry of the matrix $A^{\leq
    2k+1}(\vz,\vo)$.

  We claim that, for all \( n\ge0 \) and $0\leq i \leq k$, (with \( k\ge1 \) fixed)
  \begin{align}\label{eq: conj 50 connection 1}
    d_{n+1,i}-d_{n+1,i-1}=a_{n,i}+(-1)^{i-1}a_{n,i-1}.
  \end{align}
  Note that if \( i=0 \), we have \( d_{n+1,0}=a_{n,0} \), which is equivalent
  to \eqref{eq:e0e0=e0v}. Thus it suffices to prove the claim.

  To prove the claim \eqref{eq: conj 50 connection 1} we proceed by induction on
  $n$. The base case $n=0$ is easily checked by 
  \begin{align*}
    (d_{1,0},\dots,d_{1,2k+1})^{T}
    &= (A^{\leq 2k+1}(\vz,\vo))^{1} v=(1,2,\dots,2,1)^T,\\
    (a_{0,0},\dots,a_{0,k})^{T}
    &=  \epsilon_0 = (1,0,\dots,0)^T.
  \end{align*}

  Now assume \eqref{eq: conj 50 connection 1} is true for $n$ and consider the
  case $n+1$. Suppose $ 0\leq i \leq k-1$. By \eqref{eq:d1} and the
  induction hypothesis, 
  \begin{align*}
    d_{n+2,i}-d_{n+2,i-1}  &= (d_{n+1,i-1}+d_{n+1,i+1})-(d_{n+1,i-2}+d_{n+1,i})\\
   & =a_{n,i-1}+(-1)^{i-2}a_{n,i-2}+a_{n,i+1}+(-1)^{i}a_{n,i}.
  \end{align*}
  On the other hand, by \eqref{eq:a1},
  \begin{align*}
   & a_{n+1,i}+(-1)^{i-1}a_{n+1,i-1} \\
    &=(-a_{n,i-1}+(-1)^{i}2a_{n,i}+a_{n,i+1})+(-1)^{i-1}(-a_{n,i-2}+(-1)^{i-1}2a_{n,i-1}+a_{n,i})\\
    & =a_{n,i-1}+(-1)^{i-2}a_{n,i-2}+a_{n,i+1}+(-1)^{i}a_{n,i}.
\end{align*}
Thus \(d_{n+2,i}-d_{n+2,i-1}=a_{n+1,i}+(-1)^{i-1}a_{n+1,i-1} \). Suppose $i=k$.
By \eqref{eq:d1}, \eqref{eq:d2}, and the induction hypothesis, 
\begin{align*}
  d_{n+2,k}-d_{n+2,k-1}&=(d_{n+1,k-1}+d_{n+1,k+1})-(d_{n+1,k-2}+d_{n+1,k})\\
                       &=d_{n+1,k-1}-d_{n+1,k-2}\\
  &=a_{n,k-1}+(-1)^{k-2}a_{n,k-2}.
\end{align*}
On the other hand, by \eqref{eq:a2}, 
\begin{align*}
  &a_{n+1,k}+(-1)^{k-1}a_{n+1,k-1}\\
  &=(-a_{n,k-1}+(-1)^{k}a_{n,k})+
  (-1)^{k-1}(-a_{n,k-2}+(-1)^{k-1}2a_{n,k-1}+a_{n,k})\\
  &=a_{n,k-1}+(-1)^{k-2}a_{n,k-2}.
\end{align*}
Thus we also have \(d_{n+2,i}-d_{n+2,i-1}=a_{n+1,i}+(-1)^{i-1}a_{n+1,i-1} \).
This settles \eqref{eq: conj 50 connection 1} by induction
and the proof is completed.
\end{proof}

  In order to prove Lemma~\ref{lemma: conj 50 connection 2} we need the
  following two lemmas.

\begin{lem}\label{lemma: Alt PV connection}
  We have
  \begin{equation*}
    |\Alt_{n}^{\leq k+1}| =  \epsilon_0^{T}(A')^n v,
  \end{equation*}
  where $v=\sum_{s=0}^{2k+1}\epsilon_s$ and $A'=(A'_{i,j})_{i,j=0}^{2k+1}$ is the
  matrix defined by
  \[A'_{i,j} =\begin{cases}
      1 &\mbox{if \( i \equiv 0 \pmod 2, j \equiv 1 \pmod 2 \), and \( i<j \)}, \\
      1 &\mbox{if \( i \equiv 1 \pmod 2, j \equiv 0 \pmod 2 \), and \( i>j \)},\\
     0 &\text{otherwise}.
    \end{cases}
    \]
\end{lem}
\begin{proof}
  From the definition of $A'$, the value $\epsilon_0^{T}(A')^n \epsilon_s$
  equals the number of sequences $(a_1,\dots,a_n)$ satisfying the following
  three conditions:
\begin{enumerate}
  \item $0 \leq a_i \leq 2k+1$ and  $a_i \equiv i \pmod 2$,
  \item $a_1>a_2<a_3>a_4<\cdots$,
  \item $a_n=s$.
\end{enumerate}
So $\epsilon_0^{T}(A')^n v$ counts the number of sequences $(a_1,\dots,a_n)$ satisfying the conditions (1) and (2).
Using the bijection in Proposition \ref{prop:PV=Alt},
such sequences are in bijection with the elements of $\Alt_{n}^{\leq k+1}$.
\end{proof}

\begin{lem}\label{lemma: inverse of A}
  We have 
\[
\left.  R^{(k)}\left(A^{\leq k}(\vb,\vla)\right)\right|_{\vb=\vb^{(k)}, \vla=-\vo}=B^{-1}, 
\]
where $B=(B_{i,j})_{i,j=0}^{k}$ is the matrix defined by
\[
  B_{i,j}=
    (-1)^{\flr{\frac{k-i}{2}}+\flr{\frac{k+1-j}{2}}}\left(k+1-\max(i,j)\right).
\]
\end{lem}
\begin{proof}Denote $\bar{A}= \left(\bar{A}_{i,j}\right)_{i,j=0}^{k} = \left(R^{(k)}(A^{\leq
      k}(\vb,\vla))\right)\vert_{\vb=\vb^{(k)}, \vla=-\vo}$, in other words,
\begin{align*}
  \bar{A}_{i,j} =\begin{cases}
    (-1)^{k} &\mbox{if \(  i=j=0 \)},\\
    (-1)^{k-i}2 &\mbox{if \(  i=j\ge1 \)},\\
    -1 &\mbox{if \(  i=j+1 \)},\\
    1 &\mbox{if \(  i=j-1 \)},\\
    0 &\mbox{otherwise}.
   \end{cases}
  \end{align*}

  We must show $\left(B\bar{A}\right)_{i,j}=\delta_{i,j}$ for \( 0\le
  i,j\le k \), where \( \delta_{i,j}=1 \) if \( i=j \) and \( \delta_{i,j}=0 \)
  otherwise. To this end we consider the following three cases.
    
  Firstly, suppose $j=0$. Then $\left(B\bar{A}\right)_{i,0}=
  (-1)^{k}B_{i,0}-B_{i,1}$. Since
  $k=\flr{\frac{k}{2}}+\flr{\frac{k+1}{2}}$,
  \begin{align*}
    \left(B\bar{A}\right)_{i,0}
    &=(-1)^{k}(-1)^{\flr{\frac{k-i}{2}}+\flr{\frac{k+1}{2}}}(k+1-i)-(-1)^{\flr{\frac{k-i}{2}}+\flr{\frac{k}{2}}}(k+1-\max(i,1))\\
    &=(-1)^{\flr{\frac{k-i}{2}}+\flr{\frac{k}{2}}}\left( -i+\max(i,1)  \right)=\delta_{i,0}.
  \end{align*}

  Secondly, suppose $1\le j\le k-1$. Then \begin{equation*}\left(B\bar{A}\right)_{i,j}=
  B_{i,j-1}+(-1)^{k-j}2B_{i,j}-B_{i,j+1}.
  \end{equation*}
  Letting \( s=(-1)^{\flr{\frac{k-i}{2}}+\flr{\frac{k-j}{2}}} \), we have
  \begin{align*}
    B_{i,j-1} &=   -s(k+1-\max(i,j-1)),\\
    (-1)^{k-j}2B_{i,j} &= (-1)^{\flr{\frac{k-j}{2}}+\flr{\frac{k+1-j}{2}}}2B_{i,j} =   2s(k+1-\max(i,j)),\\
    -B_{i,j+1} &=   -s(k+1-\max(i,j+1)).
  \end{align*}
  Thus
  \[
    \left(B\bar{A}\right)_{i,j}
    = s(\max(i,j-1)-2\max(i,j)+\max(i,j+1))  = \delta_{i,j}.
  \]

  Finally, suppose $j=k$. In this case we have
  \begin{align*}
    \left(B\bar{A}\right)_{i,k}
    &=B_{i,k-1}+2B_{i,k}\\
    &=(-1)^{\flr{\frac{k-i}{2}}+1}(k+1-\max(i,k-1))
      +2(-1)^{\flr{\frac{k-i}{2}}}(k+1-k)\\
    &=(-1)^{\flr{\frac{k-i}{2}}}(\max(i,k-1)-k+1) = \delta_{i,k}.
  \end{align*}
  Therefore $\left(B\bar{A}\right)_{i,j}=\delta_{i,j}$ for all \( 0\le i,j\le k
  \) and the lemma follows.
\end{proof}

Now we are ready to prove Lemma~\ref{lemma: conj 50 connection 2}.

  \begin{proof}[Proof of Lemma~\ref{lemma: conj 50 connection 2}]
    Recall the matrices $A'$ and $B$ in Lemmas \ref{lemma: Alt PV connection}
    and \ref{lemma: inverse of A} respectively. By these lemmas and
    Proposition~\ref{prop: matrix interpretation of negative moment}, it is
    enough to show
    \begin{equation}\label{eq:A'=B}
      \epsilon_0^{T} (A')^{n} v=(-1)^{kn}\epsilon_0^{T}B^{n}\epsilon_0,
    \end{equation}
    where $v=\sum_{s=0}^{2k+1}\epsilon_s$.

    Denote $(A')^{n} v=(a_{n,0},\dots,a_{n,2k+1})^{T}$ and
    $B^{n}\epsilon_0=(b_{n,0},\dots,b_{n,k})^{T}$. We have
    $a_{n,i}=a_{n,2k+1-i}$ for $0\leq i\leq k$ due to the symmetry of the
    matrix $A'$.

    We claim that for $0\leq i \leq k$,
    \begin{equation}\label{eq: conj 50 connection 2}
      a_{n,i}-a_{n,i-1}=
      \begin{cases}
        (-1)^{n-1+\flr{\frac{i-1}{2}}}b_{n,i} & \mbox{if \( k \equiv 1 \pmod 2 \)}, \\
        (-1)^{\flr{\frac{i}{2}}}b_{n,i}  & \mbox{if \( k \equiv 0 \pmod 2 \)},
      \end{cases}
    \end{equation}
    where $a_{n,i}=b_{n,i}=0$ if $i<0$. Observe that if \( i=0 \) in \eqref{eq:
      conj 50 connection 2}, we have \( a_{n,0}=(-1)^{n} b_{n,0} \) if \(
    k\equiv 1 \pmod 2 \) and \( a_{n,0}=b_{n,0} \) if \( k \equiv 0 \pmod 2 \),
    which is equivalent to \eqref{eq:A'=B}. Hence it suffices to prove
    \eqref{eq: conj 50 connection 2}.

    To prove the claim \eqref{eq: conj 50 connection 2}, we proceed by induction
    on $n$. The base case $n=0$ is trivial. Now assume \eqref{eq: conj 50
      connection 2} is true for $n$ and consider the case $n+1$. We will only
    prove the case when $k$ is even because the other case can be proved similarly.
  
    For \( 0\le i\le k \), using the symmetry $a_{n,i}=a_{n,2k+1-i}$, we get
  \begin{align*}
    a_{n+1,i}=\sum_{j=0}^{2k+1}A'_{i,j}a_{n,j}=\sum_{j=0}^{k}\left(A'_{i,j}+A'_{i,2k+1-j}\right)a_{n,j},
  \end{align*}
which implies
\[
  a_{n+1,i}=
  \begin{cases}
    \sum_{j=0}^{k}a_{n,j}-\sum_{j=0}^{{i}/{2}}a_{n,2j-1}
    & \mbox{if \( i \equiv 0 \pmod 2 \)},  \\
    \sum_{j=0}^{(i-1)/{2}}a_{n,2j}
    & \mbox{if \( i \equiv 1 \pmod 2 \)} .
  \end{cases}
\]
Therefore we have
\begin{equation}\label{eq:a-a=a}
  a_{n+1,i}-a_{n+1,i-1}=(-1)^{i}\sum_{j=i}^{k}a_{n,j}.
\end{equation}
Using the induction hypothesis, we sum the identities \eqref{eq: conj 50
  connection 2}, where the index \( i \) takes the values \( i,i-1,\dots,1 \),
to obtain
\begin{equation}\label{eq:a=sumb}
  a_{n,i}=\sum_{j=0}^{i}(-1)^{\flr{\frac{j}{2}}}b_{n,j} .
\end{equation}

Now we compare both sides of \eqref{eq: conj 50 connection 2} for the case \(
n+1 \). By \eqref{eq:a-a=a} and \eqref{eq:a=sumb}, we have
\begin{equation}\label{eq:a-a=b+b}
  a_{n+1,i}-a_{n+1,i-1}=\sum_{j=0}^{i}(-1)^{i+\flr{\frac{j}{2}}}(k+1-i)b_{n,j}
  +\sum_{j=i+1}^{k}(-1)^{i+\flr{\frac{j}{2}}}(k+1-j)b_{n,j}.
\end{equation}
On the other hand,
\[
  (-1)^{\flr{\frac{i}{2}}} b_{n+1,i}
  = (-1)^{\flr{\frac{i}{2}}} \sum_{j=0}^{k} B_{i,j}b_{n,j}
  = \sum_{j=0}^{k}(-1)^{\flr{\frac{i}{2}}+\flr{\frac{k-i}{2}}+\flr{\frac{k+1-j}{2}}}
  (k+1-\max(i,j))b_{n,j},
\]
which is equal to the right-hand side of \eqref{eq:a-a=b+b} because
the assumption that \( k \) is even implies
\[
  (-1)^{\flr{\frac{i}{2}}+\flr{\frac{k-i}{2}}+\flr{\frac{k+1-j}{2}}}
  =(-1)^{\flr{\frac{i}{2}}+\flr{\frac{-i}{2}}+\flr{\frac{1-j}{2}}}
  =(-1)^{i+\flr{\frac{j}{2}}}.
\]
Therefore \eqref{eq: conj 50 connection 2} is also true for \( n+1 \). By
induction the claim is settled, which completes the proof.
\end{proof}

We finish this section by presenting an interesting consequence of
Lemma~\ref{lemma: conj 50 connection 2}.

\begin{cor}\label{cor: alt generating}
  We have
  \[
    \sum_{n\ge1} |\Alt_{n}^{\leq k+1}|x^n 
    =\cfrac{-y}{
      y-b_0- \cfrac{-1}{
        y-b_1- \cfrac{-1}{
          y-b_2- \genfrac{}{}{0pt}{1}{}{\displaystyle\ddots -
            \cfrac{-1}{y-b_k}}} }},
  \]
  where $y=(-1)^k x$ and 
  \begin{align*}
    b_i=\begin{cases}
      (-1)^k  &\mbox{if \( i=0  \)}, \\
      (-1)^{k-i}2  &\mbox{if \( 1 \leq i \leq k \)}.
    \end{cases}
  \end{align*}
\end{cor}
\begin{proof}
  This is immediate from Proposition \ref{prop:mu=cont} and Lemma~\ref{lemma: conj 50 connection 2}.
\end{proof}

It would be interesting to find a direct combinatorial proof of
Corollary \ref{cor: alt generating}.

\section{Application of the general reciprocity theorem}
\label{sec:appl-gener-recipr}

In this section we show that the general reciprocity theorem (Theorem~\ref{thm:
  CKconjecturemain}) implies the following result of Cigler and Krattenthaler
\cite[Theorem~34]{Kratt_Hankel1}. Using this theorem we give a generalization of
a result on reverse plane partitions, which was conjectured by Morales, Pak, and
Panova \cite{MPP2} and proved independently by Hwang et al.~\cite{Hwang2019} and
Guo et al.~\cite{GZZ2019}.

\begin{thm}\cite[Theorem~34]{Kratt_Hankel1}\label{thm: thm 34 Ck}
  We have
  \begin{multline*}
    \det\left(\mu^{\leq 2k+2m-1}_{2n+2i+2j+4m-2}(\vz,\vla)\right)_{i,j=0}^{k-1}\\
    = \left( \prod_{i=1}^{k+m-1} \lambda_{2i}^{k-i}
      \prod_{i=1}^{k+m}\lambda_{2i-1}^{k-i+n+2m-1} \right)
    R^{(2k+2m-1)}\left(\det\left(\mu^{\leq 2k+2m-1}_{-2n-2i-2j}(\vz,\vla)\right)_{i,j=0}^{m-1}\right).
  \end{multline*}
\end{thm}

We note that the statement in \cite[Theorem~34]{Kratt_Hankel1} uses the change
of variables \( \lambda_{2i-1}=A_i^{-1}V_i^{-1} \) and \(
\lambda_{2i}=A_i^{-1}V_{i+1}^{-1} \).

As before let \( \vb=(b_i)_{i\ge0} \) and \( \vla=(\lambda_i)_{i\ge1} \) be
sequences of indeterminates. We define \( \lambda_0=0 \) and the following
sequences:
\begin{align*}
 \bm{b'}&=\left( b'_i \right)_{i\ge0}, \qquad b'_i=\lambda_{2i}+\lambda_{2i+1},\\
 \bm{\lambda'}&=\left( \lambda'_i \right)_{i\ge1}, \qquad \lambda'_i=\lambda_{2i-1}\lambda_{2i},\\
 \bm{b''}&=\left( b''_i \right)_{i\ge0}, \qquad b''_i=\lambda_{2i+1}+\lambda_{2i+2},\\
 \bm{\lambda''}&=\left( \lambda''_i \right)_{i\ge1},\qquad \lambda''_i=\lambda_{2i}\lambda_{2i+1}.
\end{align*}

\begin{lem}\label{lemma: dyck moztkin connection}We have
  \begin{equation*}
    \mu_{2n}^{\leq 2k-1}(\vz,\vla)=\mu_{n}^{\leq k-1}(\bm{b'},\bm{\lambda'})
    =\left.\lambda_1\mu_{n-1}^{\leq k-1}(\bm{b''},\bm{\lambda''})\right|_{\lambda_{2k}=0}.
  \end{equation*}
\end{lem}
\begin{proof}
  This can be proved by the same method in the proof of
  \cite[Proposition~4.2]{CKS}.
\end{proof}

\begin{lem}\label{lemma: determinant b' lam'}
  We have
  \begin{equation*}
    \det\left(A^{\leq k-1}(\bm{b'},\bm{\lambda'})\right)=\prod_{i=1}^{k}\lambda_{2i-1}.
  \end{equation*}
\end{lem}
\begin{proof}
  Expanding the determinant along the last row gives a simple recurrence for the
  left-hand side. Then the lemma follows easily by induction.
\end{proof}

\begin{lem}\label{lemma: b'' lam''}We have
  \begin{equation*}
    \left.\left(R^{(k-1)}\left( \mu_{-n}^{\leq k-1}(\vb,\vla) \right)\right)\right\vert_{\vb=\bm{b'},\vla=\bm{\lambda'}}
    =R^{(2k-1)}\left(\lambda_1^{-1}\mu^{\leq 2k-1}_{-2n+2}(\vz,\vla)\right).
  \end{equation*}
\end{lem}
\begin{proof}
 It is easy to see that  
   \begin{equation*}
    \left.\left(R^{(k-1)}\left( \mu_{n}^{\leq k-1}(\vb,\vla) \right)\right)\right\vert_{\vb=\bm{b'},\vla=\bm{\lambda'}}
    =\left.R^{(2k-1)}\left(\mu^{\leq k-1}_{n}(\bm{b''},\bm{\lambda''})\right)\right|_{\lambda_{2k}=0},
  \end{equation*}
  where in the right-hand side \( \mu^{\leq k-1}_{n}(\bm{b''},\bm{\lambda''}) \)
  is a polynomial in \( \lambda_i \)'s and the operator \( R^{(2k-1)} \)
  replaces \( \lambda_i \) to \( \lambda_{2k-i} \).
  Thus, by Lemma \ref{lemma: dyck moztkin connection}, we obtain
  \begin{equation*}
    \left.\left(R^{(k-1)}\left( \mu_{n}^{\leq k-1}(\vb,\vla) \right)\right)\right\vert_{\vb=\bm{b'},\vla=\bm{\lambda'}}
    =R^{(2k-1)}\left(\lambda_1^{-1}\mu^{\leq 2k-1}_{2n+2}(\vz,\vla)\right).
  \end{equation*}
  Extending the both sides to the negative indices completes the proof. 
\end{proof}

\begin{proof}[Proof of Theorem~\ref{thm: thm 34 Ck}]
  We put $\vb=\bm{b'}$ and $\vla=\bm{\lambda'}$ in Theorem~\ref{thm:
    CKconjecturemain}. By Lemmas~\ref{lemma: dyck moztkin connection},
  \ref{lemma: determinant b' lam'} and \ref{lemma: b'' lam''} we get
  \begin{multline*}
    \det\left(\mu^{\leq 2k+2m-1}_{2n+2i+2j+4m-4}(\vz,\vla)\right)_{i,j=0}^{k-1}\\
    =\left(\prod_{i=1}^{k+m-1} \lambda_{2i-1}^{k-i}\lambda_{2i}^{k-i}
      \prod_{i=1}^{k+m}\lambda_{2i-1}^{n+2m-2}\right)
    R^{(2k+2m-1)}\left(\det\left(\lambda_1^{-1}\mu^{\leq 2k+2m-1}_{-2n-2i-2j+2}(\vz,\vla)\right)_{i,j=0}^{m-1}\right).
  \end{multline*}
  Pulling out the factor \( \lambda_1^{-1} \) in the determinant and replacing
  \( n \) by \( n+1 \) gives the desired equation.
\end{proof}

Now we will give an application of Theorem~\ref{thm: thm 34 Ck} to reverse plane
partitions. We will use the definitions of partitions and reverse plane
partitions in Hwang et al.~\cite{Hwang2019}.

We denote by \( \RPP(\lambda/\mu) \) the set of reverse plane partitions of
shape \( \lambda/\mu \). We also denote by \( \RPP^{\le k}(\lambda/\mu) \) the
set of reverse plane partitions in \( \RPP(\lambda/\mu) \) whose entries are
contained in \( \{0,\dots,k\} \). For a reverse plane partition \( T \), we
define \( |T| \) to be the sum of all entries in \( T \).

Recall that \( \Alt_{2n+1} \) is the set of sequences \( a_1\le a_2 \ge a_3\le
\cdots\ge a_{2n+1} \) of positive integers and \( \Alt_{2n+1}^{\le k} \) is the
set of sequences in \( \Alt_{2n+1} \) whose entries are contained in \(
\{1,\dots,k\} \). We define \( \overline{\Alt}_{2n+1} \) to be the set of
sequences \( a_1\ge a_2 \le a_3\ge \cdots\le a_{2n+1} \) of positive integers
and define \( \overline{\Alt}_{2n+1}^{\le k} \) to be the set of sequences in \(
\overline{\Alt}_{2n+1} \) whose entries are contained in \( \{1,\dots,k\} \).
For a sequence \( s=(a_1,\dots,a_{2n+1}) \) in \( \Alt_{2n+1} \) or \(
\overline{\Alt}_{2n+1} \), let \( |s|=a_1+\dots+a_{2n+1} \).

Morales, Pak, and Panova \cite{MPP2} conjectured the following identity, which
was proved independently by Hwang et al.~\cite{Hwang2019} and Guo et
al.~\cite{GZZ2019}:
\begin{align}\label{eq:Hwang_et_al}
  \sum_{\pi \in \operatorname{RPP}\left(\delta_{n+2 m} / \delta_{n}\right)} q^{|\pi|}=q^{-\frac{m(m+1)(6 n+8 m-5)}{6}}
  \det\left(\sum_{s\in\overline{\Alt}_{2n+2i+2j+1}}q^{|s|}\right)_{i, j=0}^{m-1},
\end{align}
where \( \delta_n=(n-1,\dots,0) \).
In \cite[Theorem~1.2]{Hwang2019} the matrix entries are generating functions for
alternating sequences of nonnegative integers, see \cite[(20)]{Hwang2019},
whereas \eqref{eq:Hwang_et_al} uses alternating sequences of positive integers.
It is easy to check that the two statements are equivalent.

In \cite[Theorem~36]{Kratt_Hankel1} Cigler and Krattenthaler gave an equivalent
statement of Theorem~\ref{thm: thm 34 Ck} using trapezoidal arrays. They also
found a bijection between trapezoidal arrays and bounded plane partitions. Using
their bijection, a simple connection between bounded plane partitions and
bounded reverse plane partitions, a simple connection between \( \Alt^{\le
  k+m}_{2n+1} \) and \( \overline{\Alt}^{\le k+m}_{2n+1} \), and the change of
variables \( V_i\mapsto V_{k+m+1-i} \) and \( A_i\mapsto A_{k+m+1-i} \), we can
restate \cite[Theorem~36]{Kratt_Hankel1} as follows.

\begin{thm}\label{thm:RPP}\cite[Theorem~36 (restated)]{Kratt_Hankel1} 
  We have
  \begin{align}\label{eq:PP_Alt}
    \sum_{T\in\RPP^{\le k}(\delta_{n+2m}/\delta_{n})} \wt(T)
    =\det\left( \sum_{t\in\overline{\Alt}_{2n+2i+2j+1}^{\le k+m}} \wt(t) \right)_{i,j=0}^{m-1},
  \end{align}
  where for \( T\in\RPP^{\le k}(\delta_{n+2m}/\delta_{n}) \) and
  \( t=(t_1,\dots,t_{2p+1})\in \overline{\Alt}^{\le k+m}_{2p+1} \),
\begin{align*}
    \wt(t)&=V_{t_1}V_{t_3}\cdots V_{t_{2n+1}} A_{t_2}A_{t_4}\cdots A_{t_{2n}},\\
    \wt(T)&=\prod_{(i,j)\in\delta_{n+2m}/\delta_{n}} \wt(T(i,j)) ,\\
    \wt(T(i,j))&=
    \begin{cases}
      A_{T(i,j)+\flr{(i+j-n+1)/2}} &\text{if \( i+j-n \) is odd},\\
      V_{T(i,j)+\flr{(i+j-n+1)/2}} &\text{if \( i+j-n \) is even}.
    \end{cases}
  \end{align*}
\end{thm}

Substituting \( V_i =A_i= q^{i} \) in the equation
\eqref{eq:PP_Alt} gives the following corollary.
\begin{cor}\label{cor:rpp_q=det_q}
 We have
  \begin{align*}
    \sum_{S\in\RPP^{\le k}(\delta_{n+2m}/\delta_{n})} q^{|S|}
    = q^{-\frac{m(m+1)(6 n+8 m-5)}{6}}
    \det\left(\sum_{s\in\overline{\Alt}_{2n+2i+2j+1}^{\le k+m}} q^{|s|} \right)
      _{i,j=0}^{m-1}.
  \end{align*}
\end{cor}
If \( k\rightarrow\infty \) in Corollary~\ref{cor:rpp_q=det_q}, we get
\eqref{eq:Hwang_et_al}.

\section{Negative moments of Laurent biorthogonal polynomials}
\label{sec:laurent}

Recall that we have combinatorial reciprocity theorems for the number of Dyck
paths of bounded height and for the number of Motzkin paths of bounded height.
Therefore it is natural to ask whether there is a reciprocity theorem for the
number of Schr\"oder paths of bounded height. In this section we study the
negative version of the number of Schr\"oder paths with bounded height and its
connection with the negative moment of Laurent biorthogonal polynomials.

The \emph{Laurent biorthogonal polynomials} \(
(L_n(x))_{n\ge0} \) can be defined by a three-term recurrence
\begin{equation}\label{eq:LOP}
  L_{n+1}(x) = (x-b_n)L_n(x) -a_nx L_{n-1}(x),\quad n\ge0, \qquad
  L_{-1}(x)=0,L_0(x)=1,
\end{equation}
for some sequences \( \vb=(b_n)_{n\ge0} \) and \( \va=(a_n)_{n\ge1} \). To
emphasize sequences \( \vb \) and \( \va \) we will write the polynomials \(
L_n(x) \) as \( L_n(x;\vb,\va) \). There is a unique linear functional \( \LL \)
on the space of Laurent polynomials such that \( \LL(1)=1 \) and
\[
  \LL\left(L_m(x;\vb,\va)\cdot \frac{L_n(x;\vb,\va)}{x^n}\right) = 0, \qquad 0\le m< n.
\]

Since the linear functional \( \LL \) is defined on the space of Laurent
polynomials, we have \emph{positive moments} \( \LL(x^n) \) and \emph{negative
  moments} \( \LL(x^{-n}) \). Kamioka \cite{Kamioka2007,Kamioka2014} showed that
both positive and negative moments are generating functions for Schr\"oder
paths. To state Kamioka's results, we need the following definitions.

Recall that a lattice path is a finite sequence of points in \( \ZZ\times
\ZZ_{\ge0}\).

\begin{defn}
  A \emph{Schr\"oder path} is a lattice path in which every step is an \emph{up
    step} $(1,1)$, a \emph{double-horizontal step} $(2,0)$, or a \emph{down
    step} \( (1,-1) \). The set of Schr\"oder paths from \( (0,0) \) to \( (n,0)
  \) is denoted by \( \Sch_n \).

  For given sequences \( \vb=(b_0,b_1,\dots) \) and \( \va=(a_1,a_2,\dots) \), the
  \emph{weight} \( \wt(\pi;\vb,\va) \) of a Schr\"oder path \( \pi \) is defined
  to be the product of \( b_i \) for each double-horizontal step starting at a point with
  \( y \)-coordinate \( i \) and \( a_i \) for each down step starting at a point
  with \( y \)-coordinate \( i \).
\end{defn}

Kamioka \cite{Kamioka2007,Kamioka2014} showed that the moments \( \LL(x^n) \),
\( n\in\ZZ \), of Laurent biorthogonal polynomials \( L_n(x;\vb,\va) \) are
generating functions for Schr\"oder paths: for \( n\ge0 \),
\begin{align}
  \label{eq:Kamioka1}
  \LL(x^{n}) &=  \sum_{\pi\in\Sch_{2n}} \wt(\pi;\vb,\va),\\
  \label{eq:Kamioka2}
  \LL(x^{-n-1}) &= b_0^{-1} \sum_{\pi\in\Sch_{2n}} \wt(\pi;\vb',\va'),
\end{align}
where \( \vb' \) and \( \va' \) are the sequences defined by
\begin{align}
  \label{eq:b'}
  \vb' & =(b'_i)_{i\ge0}, \qquad b'_i=b_i^{-1},\\
  \label{eq:a'}
  \va'&=(a'_i)_{i\ge1} , \qquad a'_i = a_ib_{i-1}^{-1}b_{i}^{-1}.
\end{align}

\begin{defn}
  We define the \emph{bounded moment} \( \sigma_n^{\le k }(\vb,\va) \) of the
  Laurent biorthogonal polynomials \( L_n(x;\vb,\va) \) by
  \[
    \sigma_n^{\le k }(\vb,\va)=  \sum_{\pi\in\Sch_{2n}^{\le k}} \wt(\pi;\vb,\va),
  \]
  where \( \Sch_{2n}^{\le k} \) is the set of Schr\"oder paths from \( (0,0) \)
  to \( (2n,0) \) that stay weakly below the line \( y=k \). 
\end{defn}

Note that if the sequence \( (\sigma_n^{\le k }(\vb,\va))_{n\ge0} \) satisfies a
homogeneous linear recurrence relation, then its negative version \(
(\sigma_{-n}^{\le k }(\vb,\va))_{n\ge1} \) is defined. By definition we have
\[
  \LL(x^n) = \lim_{k\to\infty} \sigma_n^{\le k }(\vb,\va), \qquad n\ge0.
\]
The goal of this section is to prove that the negative moment \( \LL(x^{-n}) \)
is also the limit of the negative version of \( (\sigma_n^{\le k }(\vb,\va))_{n\ge0}
\).

By specializing the results of Kim and Stanton~\cite[Corollary~5.4 and
Propositions~5.5]{kimstanton:R1} on orthogonal polynomials of type \( R_I \) we
obtain the following.

\begin{prop}\label{prop:kimstanton_con_fra}
  We have
    \begin{align*}
      \sum_{n\ge0} \sigma_{n}^{\le k}(\vb,\va) x^n
      =\frac{\delta P^*_{k}(x;\vb,\va)}{P^*_{k+1}(x;\vb,\va)}
        =\cfrac{1}{
          1-b_0x- \cfrac{a_1x}{
          1-b_1x- \cfrac{a_2x}{
          1-b_2x- \genfrac{}{}{0pt}{1}{}{\displaystyle\ddots -
          \cfrac{a_kx}{1-b_k x}}} }},
    \end{align*} 
    where \( \delta P^*_{k}(x;\vb,\va) \) and \(  P^*_{k+1}(x;\vb,\va) \)
    are defined similarly as in Definition~\ref{defn:inverted}.
\end{prop}

The following proposition can be proved similarly as
Propositions~\ref{prop:well-defined} and \ref{prop:well-defined2}.

\begin{prop}
  If \( P_{k+1}(0;\vb,\va)\ne0 \), then \( \sigma_{-n}^{\le k}(\vb,\va) \) is
  well defined for \( n\ge1 \). In particular, if \( b_i\ne0 \) for all \( i\ge0
  \), then \( \sigma_{-n}^{\le k}(\vb,\va) \) is well defined for \( n\ge1 \).
\end{prop}

By the same argument as the one in the proof of Proposition~\ref{prop:mu=cont}, we obtain
the generating function for \( \sigma^{\le k}_{-n}(\vb,\va)\) as follows.

\begin{prop}\label{prop:sigma-n}
  If \( (\sigma_{-n}^{\le k}(\vb,\va))_{n\ge1} \) is defined, we have
  \begin{align*}
    \sum_{n\ge1} \sigma_{-n}^{\le k}(\vb,\va) x^n
    =-\frac{x \delta P_{k}(x;\vb,\va)}{P_{k+1}(x;\vb,\va)}
    =\cfrac{x}{
    b_0-x- \cfrac{a_1x}{
    b_1-x- \cfrac{a_2x}{
    b_2-x- \genfrac{}{}{0pt}{1}{}{\displaystyle\ddots -
    \cfrac{a_kx}{b_k-x}}} }}.
  \end{align*}
\end{prop}

Using Proposition~\ref{prop:sigma-n}, we can find a combinatorial interpretation
for \( \sigma^{\le k}_{-n}(\vb,\va) \).

\begin{thm}\label{thm:sigma-n}
    Let \( n, k \) be positive integers. We have
    \[
      \sigma^{\le k}_{-n}(\vb,\va)
      = b^{-1}_0 \sum_{\pi\in \Sch^{\le k}_{n-1}} \wt(\pi;\vb',\va'),
    \]
    where \( \vb' \) and \( \va' \) are defined in \eqref{eq:b'} and \eqref{eq:a'}.
\end{thm}
\begin{proof}
    Let
    \[
      f^{\le k}_n=\sum_{\pi\in \Sch^{\le k}_{n}} \wt(\pi;\vb',\va').
    \]
    By Proposition \ref{prop:kimstanton_con_fra}, we have
    \begin{align*}
        \sum_{n\ge0} f_{n}^{\le k} x^n
        &=\cfrac{1}{
          1-b^{-1}_0x- \cfrac{a_1 b^{-1}_0 b^{-1}_1 x}{
          1-b^{-1}_1x- \cfrac{a_2 b^{-1}_1 b^{-1}_2 x}{
          1-b^{-1}_2x- \genfrac{}{}{0pt}{1}{}{\displaystyle\ddots -
          \cfrac{a_k b^{-1}_{k-1} b^{-1}_k x}{1-b^{-1}_k x}}} }} \\
        &=\cfrac{b_0}{
          b_0-x- \cfrac{a_1x}{
          b_1-x- \cfrac{a_2x}{
          b_2-x- \genfrac{}{}{0pt}{1}{}{\displaystyle\ddots -
          \cfrac{a_kx}{b_k-x}}} }}.
    \end{align*} 
Comparing this with Proposition~\ref{prop:sigma-n}, we obtain
    \[
        \sum_{n\ge1} \sigma_{-n}^{\le k}(\vb,\va) x^n
        =b^{-1}_0 x \sum_{n\ge0} f_{n}^{\le k} x^n
        =\sum_{n\ge 1} b^{-1}_0 f^{\le k}_{n-1} x^n.
    \]
    Therefore \( \sigma_{-n}^{\le k}(\vb,\va) = b^{-1}_0 f^{\le k}_{n-1} \),
    which is the desired result.
\end{proof}

Substituting \( \vb=\va=\vo \) in Theorem~\ref{thm:sigma-n}, we see that the
negative version of the number of bounded Schr\"oder paths is also the number of
bounded Schr\"oder paths.

\begin{cor}\label{cor:sch}
  Let \( s_n=|\Sch_n^{\le k}| \) for \( n\ge0 \). Then for \( n\ge1 \) we have
  \( s_{-n}=s_{n-1} \).
\end{cor}

By Theorem~\ref{thm:sigma-n} and \eqref{eq:Kamioka2}, we obtain that the
negative moments \( \LL(x^{-n}) \) are the limits of the negative versions \(
\sigma_{-n}^{\le k }(\vb,\va) \) of the bounded moments \( \sigma_{n}^{\le k
}(\vb,\va) \).

\begin{cor}
  For \( n \ge 1 \),
  \[
    \LL(x^{-n}) = \lim_{k\to\infty} \sigma_{-n+1}^{\le k }(\vb,\va).
  \] 
\end{cor}

\section*{Acknowledgments}
The authors would like to thank the anonymous referees for helpful comments.

All authors were supported by NRF grant \#2016R1A5A1008055. Jihyeug
Jang, Jang Soo Kim, and U-Keun Song were supported by NRF grant
\#2022R1A2C101100911. Donghyun Kim was supported by NRF grant
\#2022R1I1A1A01070260. Minho Song was supported by NRF grant
\#2022R1C1C2009025.

\end{document}